\documentclass[12pt]{amsart}
\usepackage{tikz}

\usepackage[pdftex]{hyperref}
\hypersetup{
colorlinks=true,
linkcolor=blue,
urlcolor=blue,
citecolor=blue
}

\usepackage[mathlines]{lineno}

\usepackage{color}

\usepackage{amsmath}
\usepackage{amssymb}
\usepackage{amsthm}
\usepackage{graphicx}
\usepackage{multicol}

\usepackage{listings}

\usepackage{enumerate}
\usepackage{float}

\newcommand{\N}{\mathbb{N}}

\newcommand{\sig}{\sigma}

\newcommand{\eps}{\varepsilon}

\newcommand{\lam}{\lambda}

\newcommand{\R}{\mathbb{R}}

\newcommand{\Z}{\mathbb{Z}}

\DeclareMathOperator{\fix}{Fix}

\DeclareMathOperator{\aut}{Aut}

\DeclareMathOperator{\spn}{Span}

\DeclareMathOperator{\col}{Col}

\theoremstyle{plain} 
\newtheorem{thm}{Theorem}[section]

\newtheorem{prop}[thm]{Proposition}

\theoremstyle{definition}
\newtheorem{definition}[thm]{Definition}

\newtheorem{example}[thm]{Example}

\theoremstyle{remark}
\newtheorem{rem}[thm]{Remark}

\floatstyle{boxed}
\newfloat{Algorithm}{tbp}{lop} 
\floatname{Algorithm}{\sc{Algorithm}}

\begin{document}


\newpage 

\thanks{\today}

\title[BLIS]
{A Bifurcation Lemma for Invariant Subspaces}

\author{John M. Neuberger}
\author{N\'andor Sieben}
\author{James W. Swift}

\email{
John.Neuberger@nau.edu,
Nandor.Sieben@nau.edu,
Jim.Swift@nau.edu}

\address{
Department of Mathematics and Statistics,
Northern Arizona University PO Box 5717,
Flagstaff, AZ 86011-5717, USA
}


\begin{abstract}
The Bifurcation from a Simple Eigenvalue (BSE) Theorem
is the foundation of steady-state bifurcation theory
for one-parameter families of functions.
When eigenvalues of multiplicity greater than one are caused by symmetry, the Equivariant Branching Lemma (EBL)
can often be applied to 
predict
the branching of solutions.
The EBL can be interpreted as the application of the BSE Theorem to a fixed point subspace.
There are functions which have invariant linear subspaces that are not caused by symmetry.
For example, networks of identical coupled cells often have such invariant subspaces.
We present a generalization of 
the EBL, where the BSE Theorem is applied to nested invariant subspaces. We call this the Bifurcation Lemma for Invariant Subspaces (BLIS). 
We give several examples of bifurcations and determine if BSE, EBL, or BLIS apply.
We extend our previous automated bifurcation analysis algorithms to use the BLIS to simplify and 
improve the detection of branches created at bifurcations.
\end{abstract}

\maketitle


\section{Introduction}

We present a Bifurcation Lemma for Invariant Subspaces (BLIS).
The BLIS proves the existence of bifurcating solution branches of a nonlinear equation $F(s,x)=0$ with a single parameter~$s$.
It applies the Bifurcation from a Simple Eigenvalue (BSE) Theorem of 
Crandall and Rabinowitz \cite{CR} to nested invariant subspaces.
The BSE Theorem is a fundamental result well known to researchers in partial differential equations (PDE).
The equally fundamental Equivariant Branching Lemma (EBL) \cite{ cicogna, GSvol1, vander}
is a powerful tool for the study of bifurcations in symmetric dynamical systems.
While symmetry in systems leads to invariant fixed point subspaces \cite{GSS},
the structure of some systems, especially coupled networks, causes additional invariant subspaces  \cite{ GSbook2023, NSS3, NSS5, NSS7, NiSS}.
In \cite{NSS3}, we call these anomalous invariant subspaces.
The BLIS can be thought of as extending the EBL
from fixed point subspaces of symmetric systems to all invariant subspaces.
The BLIS is a generalization of the Synchrony Branching Lemma, which describes some bifurcations from the fully synchronous state in coupled oscillators. See \cite{FranciGolubitskyStewart, GL_SBL, soares} and \cite[Section 18.3]{GSbook2023}.

The BLIS predicts the branching of solutions in a wide variety of applications, 
including all branching predicted by the EBL.
The BLIS lends itself to implementation as a numerical algorithm,
and provides a bridge between bifurcation theory in PDE and dynamical systems.
Figure~\ref{VennDiagram} shows the relationship between BLIS, BSE, and EBL, and references our example applications. 

Some invariant subspaces are caused by the symmetry of a dynamical system.
The topology of a network of coupled cells can introduce additional invariant subspaces.
In general, there can be even more invariant subspaces.
Our current paper does not discuss the reason for, or computation of, these invariant subspaces.
A search for a more general theory of invariant subspaces can be found in \cite{Schneider2022}.
This work
includes the Equivaroid Branching Lemma \cite[Theorem 4.2.2]{Schneider2022}, which is similar to our BLIS.

In Section \ref{sec:MainThm} we give some definitions and introduce notation, and then state and prove the BLIS and some related propositions.
In Section \ref{sec:Examples} we present numerous examples of applications of the BLIS.
We consider one-dimensional systems, coupled cell networks, and PDE.
Lastly, in Section~\ref{sec:Implementation} we give some details concerning the 
algorithms and numerical implementation we used to 
generate the bifurcation diagrams in the coupled network examples.
Algorithms for finding invariant subspaces for small networks are developed in~\cite{Aguiar&Dias, KameiLattice, NSS7, NiSS},
with implementations found at~\cite{invariantWEB, NiSS_companion}.
Our code for branch following, which is freely available at~\cite{NSS8_companion},  uses the BLIS.  It improves
the Newton method implementations
 in~\cite{bifGraph, NSS3, NSS5} which use only the fixed point subspaces of the group action. %
%
%
%
%
%
%
%

\begin{figure}
\begin{tikzpicture}
\draw (2.9,2) ellipse (1.8cm and 1cm) node[fill=white,above=.75cm] {\tiny EBL}; 
\draw[rounded corners] (0,0.5)  rectangle node[above=1.25cm,fill=white] {\tiny BLIS} (6.25,3.5);
\draw (1,2) ellipse (2cm and 1.2cm) node[fill=white,above=.75cm,left=1.25cm] {\tiny BSE};
\node at (-0.8,2) [anchor = west] {\tiny\ref{trans2InR}\vphantom{(c)}};
\node at (0.1,2.75) [anchor = west] {\tiny\ref{transInR}\vphantom{(c)}};
\node at (0.1,2.25) [anchor = west] {\tiny\ref{W5toW8}\vphantom{(c)}};
\node at (0.1,1.75) [anchor = west] {\tiny\ref{transInPDE}\vphantom{(c)}};
\node at (0.1,1.25) [anchor = west] {\tiny\ref{PDEsecondary}\vphantom{(c)}};
\node at (1.6,2.5) [anchor = west]{\tiny\ref{pitchforkInR}\vphantom{(c)}};
\node at (1.6,2.0) [anchor = west] {\tiny\ref{W1toW5andW6}(b)};
\node at (1.6,1.5)  [anchor = west] {\tiny\ref{pitchforkInPDE}\vphantom{(c)}};
\node at (3.1,2.5) [anchor = west] {\tiny\ref{lam12square}\vphantom{(c)}};
\node at (3.1,2.0) [anchor = west] {\tiny\ref{exa:lambda13}(a)};
\node at (3.1,1.5)  [anchor = west] {\tiny\ref{cube6D}(a)};
\node at (4.9,2.75) [anchor = west] {\tiny\ref{W1toW5andW6}(a)};
\node at (4.9,2.25)  [anchor = west] {\tiny\ref{W0s4}\vphantom{(c)}};
\node at (4.9,1.75) [anchor = west] {\tiny\ref{exa:lambda13}(b)};
\node at (4.9,1.25) [anchor = west] {\tiny\ref{cube6D}(b)};
\node at (6.5,2.5) [anchor = west] {\tiny\ref{foldInR}\vphantom{(c)}};
\node at (6.5,2.0) [anchor = west] {\tiny\ref{nongenericInR}\vphantom{(c)}};
\node at (6.5,1.5) [anchor = west] {\tiny\ref{cube6D}(c)};
\end{tikzpicture}
\caption{
\label{VennDiagram}
The relationship between the Bifurcation Lemma for Invariant Subspaces, Bifurcation from a Simple Eigenvalue,  and the Equivariant Bifurcation Lemma.  The numbers indicate examples.
}
\end{figure}
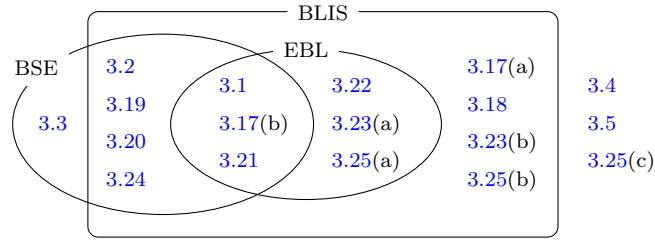

\section{The Main Results}
\label{sec:MainThm}
In this section, we give some background and state the BSE Theorem of Crandall and Rabinowitz \cite{CR}.  We then use the BSE Theorem to prove the BLIS.
\begin{definition}
\label{def:F-invariantXtoX}
Let $F: I \times X\to X$, where $I$ is an open interval and $X$ is a Banach space. A subspace $W$ of $X$ is \emph{$F$-invariant} if $F(s,W) \subseteq W$ for all $s\in I$.
\end{definition}

This extends the standard definition that $W$ is an $F$-invariant subspace for $F: X \to X$ provided $F(W) \subseteq W$.

\begin{example}
The trivial subspace $\{0\} \subseteq X$ is  $F$-invariant if and only if $F(s, 0) = 0$ for all $s \in I$.
Also, the full subspace $X$ is $F$-invariant
for any $F: I \times X \to X$.
\end{example}

We often get $F$-invariant subspaces when there is a symmetry in the system.  We recall
some definitions involving group actions as they relate to invariant subspaces.
Assume a compact Lie group $\Gamma$ acts linearly on a Banach space $X$
by the function $(\gamma,x)\mapsto \gamma x:\Gamma\times X\to X$.
We say that $F: I \times X \to X$ is $\Gamma$-\emph{equivariant} if
$F(s, \gamma x) = \gamma F(s, x)$ for all $(s, x) \in I \times X$, $\gamma \in \Gamma$.
The \emph{stabilizer of}  $x$, or the \emph{isotropy subgroup of} $\Gamma$ \emph{with respect to}
$x \in X$, is $\Gamma_x := \{ \gamma  \in \Gamma: \gamma x = x\}$.
An \emph{isotropy subgroup of} $\Gamma$ is $\Sigma = \Gamma_x$ for some $x$.
A subgroup of $\Gamma$ is not necessarily an isotropy subgroup of the group action.
For every subgroup $\Sigma$ of $\Gamma$, the \emph{fixed point subspace} of $\Sigma$ is
$$\fix(\Sigma) := \{x \in X : \sigma x = x \text { for all } \sigma \in \Sigma \}.$$ 
The \emph{point stabilizer} of a set $U \subseteq X$ is 
$$\text{pStab}(U) := \{ \gamma \in \Gamma : \gamma x = x \mbox{ for all } x \in U\}.$$
The group action on $X$ extends to an action on subsets of $X$.  If $W$ is an $F$-invariant subspace, then
$\gamma W$ is also an $F$-invariant subspace for all $\gamma \in \Gamma$.
The $\Gamma$-\emph{orbit} of an $F$-invariant subspace $W$ is
$\{ \gamma W : \gamma \in \Gamma\}$.

\begin{prop}
Given a subgroup $\Sigma$ of a compact Lie group $\Gamma$ that acts linearly on a Banach space,
$\fix(\Sigma)$ is an $F$-invariant subspace if $F$ is $\Gamma$-equivariant.
\end{prop}
\begin{proof}
First of all, $\fix(\Sigma)$ is a subspace of $X$ because $\Gamma$ acts linearly.
Assume $x \in \fix(\Sigma)$.  Then $F(s, x) = F(s,\sigma x) = \sigma F(s, x)$ for all $\sigma \in \Sigma$.  
Thus $F(s, x) \in \fix(\Sigma)$.
\end{proof}

\begin{example}
\label{exa:secondDerivative}
Let $F: C^\infty \to C^\infty$ be defined by $F(u) = u'' $, where $C^\infty$ is the space of infinitely differentiable 
real-valued functions on $\R$.
The group $\Z_2 = \{1, \gamma \}$ acts on $C^\infty$ by $(\gamma u)(t) = u(-t)$.  
Note that $F$ is $\Z_2$-equivariant because $F(\gamma u)(t)=u''(-t)=(\gamma F(u))(t)$.
The $F$-invariant fixed point subspace $\fix(\Z_2)$ is the subspace of even functions.
\end{example}

Recall that a \emph{simple $C^1$-curve} is the image of an injective $C^1$-function $p:(a,b)\to I\times X$ with non-vanishing derivative.
Note that a 
simple $C^1$-curve
does not cross itself, and it has no cusps or corners.

\begin{definition}
\label{def:branch}
Let $F: I \times X \to X$ be as in Definition~\ref{def:F-invariantXtoX}.
A \emph{solution branch of $F$}, or simply a \emph{branch}, is a simple $C^1$-curve  that is a subset of $F^{-1}(\{0\})$.
A \emph{$W$-branch} is a solution branch contained in $I \times W$, where $W$ is an $F$-invariant subspace of $X$. 
\end{definition}

\begin{example}
The function $F: \R \times \R \to \R$ defined by $F(s, x) = sx$ has two solution branches, with equations $s = 0$ and $x = 0$.  
The L-shaped curve $\{(s, 0) \mid s \geq 0\} \cup \{(0, x) \mid x \geq 0\}$ is not a solution branch because it has a corner at
$(0,0)$.
\end{example}
\begin{definition}
\label{def:bifPt}
A point $p$ in a solution branch of $F$ is a \emph{bifurcation point} if every neighborhood of $p$
in $I \times X$
contains a zero of $F$ that is not in the solution branch.
\end{definition}

Typically, a bifurcation point is the intersection of two or more branches.
In the most recent example, 
$(0,0)$ is a bifurcation point.
Note that a fold point, sometimes called a saddle-node
bifurcation, is not a bifurcation point by this definition, which follows \cite{CR}.
This definition also ignores Hopf bifurcations.

The following proposition is a consequence of the Implicit Function Theorem, and it gives conditions for which a solution $(s^*, x^*)$ to $F(s, x) = 0$ can be extended to a unique solution branch within an invariant subspace.
We call this curve of solutions the mother branch because we anticipate that the hypotheses of 
Theorem~\ref{thm:BLIS} hold.

\begin{prop}  
\label{prop:motherExists}
Let $I$ be an open interval, $X$ be a Banach space, and $F:I\times X\to X$ be a $C^{2}$-function.
Suppose $W_m$ is an $F$-invariant subspace of $X$, $(s^*, x^*) \in I \times W_m$ satisfies 
$F(s^*, x^*) = 0$, and $D_2F(s^*, x^*) |_{W_m}$ is nonsingular.  
Then there exists a neighborhood $U$ of $(s^*, x^*)$ in $I\times X$, and a
$C^{2}$-function $b_m:I_{m}\to W_{m}$ defined on an open
interval containing $s^{*}$ such that $b_m(s^{*})=x^{*}$ and
$$
F^{-1}(\{0\})  \cap (I\times W_m) \cap U = C_m ,
$$  
where
\begin{equation}
\label{eqn:Bm}
C_m := \{(s, b_m(s) ) \mid s \in I_m \}.
\end{equation}
We call $b_m$ the mother branch function and its graph $C_m$ the mother branch.
\end{prop}

\begin{proof}
This is the Implicit Function Theorem applied to the restriction $F_m: I \times W_m \to W_m$ 
of $F$ to $I \times W_m$.
The only subtlety is that $(I\times W_m)\cap U$ is a neighborhood of $(s^*, x^*)$ in $I \times W_m$.
\end{proof}

\begin{rem}
If $W_m = \{  0 \}$ is an invariant subspace, then $D_2 F(s^*, x^*)|_{W_m}$ is vacuously nonsingular.  Here the mother branch function
is $b_m(s) = 0$ for all $s \in I$, and the mother branch is called the \emph{trivial branch}.
\end{rem}

The BLIS describes a bifurcation of the mother branch of
Proposition~\ref{prop:motherExists} at $(s^*, x^*)$ to a daughter branch.
This is a bifurcation from a simple eigenvalue for the restricted function
$F_d: I \times W_d \to  W_d$, where $W_m \subsetneq W_d$ are the invariant subspaces corresponding to the mother and daughter branches, respectively.
To this end, we first give a slight modification of the BSE Theorem \cite[Theorem 1.7]{CR}.
Motivated by Definition~\ref{def:F-invariantXtoX}, we assume $F: I_m \times X \to X$,
while \cite{CR} assumes $F: (-1,1) \times X \to Y$.
In our theorem, the bifurcation is at $(s^*, 0)$, whereas \cite{CR} assumes the bifurcation point is $(0,0)$.

\begin{thm}[BSE, Theorem 1.7 of \cite{CR}]
\label{thm:CR1.7}
Let $X$ be a Banach space and $I_m$ an 
open
interval containing $s^*$.  
Let $\tilde F: I_m \times X \to X$ be a $C^2$-function with these properties:

(a) $\tilde F(s, 0) = 0$ for $s \in I_m$.

(b) The partial derivatives $D_1 \tilde F$, $D_2 \tilde F$, $D_1 D_2 \tilde F$, and $D_2^2 \tilde F$ exist and are continuous.

(c) $N(D_2\tilde F(s^*, 0))$ is one-dimensional, and $R(D_2\tilde F(s^*, 0))$ has codimension 1.

(d) $D_1D_2 \tilde F(s^*, 0) (x_0) \not \in R(D_2\tilde F(s^*, 0))$, 
where $N(D_2\tilde F(s^*, 0))= \spn(\{x_0\}).$

\noindent
If $Z$ is any complement of $\spn(\{x_0\})$ in $X$, then there is a neighborhood $U$ of $(s^*, 0)$ in $\R \times X$, an interval $(-a,a)$,
and 
$C^1$-functions
$\varphi: (-a,a) \to \R$, $\psi:(-a,a) \to Z$ such that $\varphi(0) = s^*$, $\psi(0) = 0$ and
$$
\tilde F^{-1}(\{0\}) \cap U = \{ (\varphi(\alpha), \alpha x_0 + \alpha \psi(\alpha)) : |\alpha| < a\} \cup \{(s, 0) : (s, 0) \in U\}.
$$
\end{thm}

\begin{rem}
Theorem~\ref{thm:CR1.7} is not exactly as originally stated in \cite[Theorem 1.7]{CR}.  
Our statement is convenient for our purposes and is an easy consequence of the original result,
as now we explain.

First, we scale the parameter $s \in I_m$ to $t \in (-1,1)$.
Choose $\eps > 0$ such that 
$(s^*-\eps, s^* + \eps) \subseteq I_m$.
We apply \cite[Theorem 1.7]{CR} to
$F:(-1,1)\times X \to X$  defined by $F(t, x) = \tilde F( s^* + \eps t , x)$.
%
Hypothesis (b) in \cite[Theorem 1.7]{CR} does not require $F_{xx} = D_2^2F$ to be continuous, and the conclusion is 
that $\varphi$ and $\psi$ are continuous.
In a remark after the theorem they state that if $F_{xx}$ is continuous, then $\varphi$ and $\psi$ are $C^1$, and this
is the result stated in Theorem~\ref{thm:CR1.7}.
\end{rem}

We now present our main result, which applies the BSE to nested invariant subspaces.  
Recall from Proposition~\ref{prop:motherExists} that the mother branch includes 
$\{(s, b_m(s)) : s \in I_m\}$, with  $b_m(s) \in W_m$ and $F(s, b_m(s)) = 0$ for $s \in I_m$, and $b_m(s^*) = x^*$.

\begin{thm}[Bifurcation Lemma for Invariant Subspaces]
\label{thm:BLIS}
Assume that the hypotheses of Proposition~\ref{prop:motherExists} hold. 
Let $b_m$ and $C_m$ be the mother branch function and mother branch.
Assume
$W_{m}\subsetneq W_{d}\subseteq X$ 
be nested $F$-invariant subspaces. 
Let $F_d:I\times W_d \to W_d$ be the restriction of $F$.
Assume $J:I_{m}\to L(W_{d})$, $J(s):=D_{2}F_d(s,b_m(s) )$
satisfies the following conditions:

\emph{(a)} $N(J(s^{*}))$ is one-dimensional, and $R(J(s^*))$ has codimension one.

\emph{(b)} $J'(s^{*})(x_0)\notin R(J(s^{*}))$, where $N(J(s^*)) = \spn(\{x_0\})$.

\noindent 
If $Z$ is any complement of $\spn(\{x_0\})$ in $W_d$, then
there is a neighborhood $U$ of $(s^*, x^*)$ in $I\times X$ and $C^1$-functions
$\varphi: (-a, a) \to \R$ and $\psi:(-a, a) \to W_d$ with $\varphi(0) = s^*$ and $\psi(0) = 0$ such that 
$$
F^{-1}(\{0\}) \cap (I \times W_d)\cap U = C_m \cup C_d,
$$
where
$$
C_d := \{(\varphi(\alpha), b_m(\varphi(\alpha)) + \alpha x_0 + \alpha \psi(\alpha)  ) : |\alpha| < a \}
$$   
is the so-called daughter branch. Furthermore, $C_m \cap C_d = \{(s^*, x^*)\}$.
That is,
there are exactly two branches in $(I \times W_d)\cap U$ that contain $(s^*, x^*)$; the mother branch $C_m$ and the daughter branch $C_d$.
\end{thm}

\begin{proof}
Since $W_d$ is $F$-invariant, the restriction $F_d: I \times W_d \to W_d$ is well-defined.
We will show that the hypotheses of Theorem~\ref{thm:CR1.7}
apply to the function $\tilde F: I_m \times W_d \to W_d$ defined by
$$
\tilde F(s, x) = F_d(s,  b_m(s) + x).
$$
Condition (a) holds, since $\tilde F(s, 0) = F_d(s, b_m(s)) = 0$, by the definition of the mother branch function $b_m$.  
The restriction to $F_d$ is allowed since $b_m(s) \in W_m \subseteq W_d$.
Condition (b) holds: $\tilde F$ is $C^2$ since $F$ is $C^2$ and $b_m$ is $C^2$.
Conditions (c) and (d) require the computation of $D_2\tilde F(s^*,0)$
and $D_1 D_2 \tilde F(s^*,0)$.  
The Jacobian is 
$$
 D_2\tilde F(s, x) = D_2 F_d(s, b_m(s) + x),
$$
 and the stability of the solution at any $s$ on the mother branch is determined by
\begin{equation}
\label{eqn:Jac}
 J(s) := D_2 F_d(s, b_m(s)) = D_2 \tilde F(s, 0).
\end{equation}
While $J(s)$ is defined in terms of the first expression in Equation~(\ref{eqn:Jac}), we will use the second form of the expression for
the remainder of this proof.
Condition (a) of Theorem~\ref{thm:BLIS}  concerns
 $$
 J(s^*) = D_2\tilde F(s^*, 0)
 $$
 and is equivalent to Condition (c) of Theorem~\ref{thm:CR1.7}.
 Condition (b) of Theorem~\ref{thm:BLIS}  concerns
 $$
J'(s) = 
 D_1 D_2 \tilde F (s, 0),
 $$
and is equivalent to Condition (d) of Theorem~\ref{thm:CR1.7}.

Thus, Theorem~\ref{thm:CR1.7} holds for the function $\tilde F$, and for the functions $\varphi$ and $\psi$ defined in that theorem, 
$$
\tilde F(\varphi(\alpha), \alpha x_0 + \alpha \psi(\alpha)) = 0,
$$
for $| \alpha | < a$.
By the definition of $\tilde F$ in terms of $F_d$, 
$$
F_d(\varphi(\alpha), b_m(\varphi(\alpha)) + \alpha x_0 + \alpha \psi(\alpha)) = 0,
$$
for $| \alpha | < a$, and the daughter branch is given parametrically as desired.
Theorem~\ref{thm:CR1.7} ensures that the mother and daughter branch are the only local branches in $I\times W_d$.
Finally, the point on the daughter branch $C_d$ when $\alpha = 0$ is $(\varphi(0), b_m(\varphi(0)) ) = (s^*, x^*)$.
There is no other point on the daughter branch in $I \times W_m$, since $b_m(s) \in W_m$, and $x_0 \not\in W_m$, and
$\psi(\alpha)$ is in $Z$.  Thus, the only point of intersection of $C_m$ and $C_d$ is
$(s^*, x^*)$.
 \end{proof}

\begin{rem} 
\label{rem:BLISremarks}
We give several observations.
\begin{enumerate} 
\item  
Condition (b) in Theorem~\ref{thm:BLIS} typically holds in applications, but it is difficult to check when 
the mother branch function is not known explicitly.
This nondegeneracy condition implies that a simple
eigenvalue of $J(s)$ crosses 0 with nonzero speed at $s = s^*$.  This leads to an informal statement of the BLIS:  

{\it Suppose $W_m \subsetneq W_d$ are invariant subspaces of $F:\R \times X \to X$, with $F(s^*, x^*) = 0$ for $x^* \in W_m$.
Let $K = N(D_2 F(s^*, x^*))$.
If 
$$K \cap  W_m= \{0\} \text{  and } \dim (K \cap W_d) = 1,
$$
then generically $F(s,x) = 0$ has exactly two solution branches in $\R \times W_d$ that contain $(s^*, x^*)$;
the mother branch which is in $\R \times W_m$, and the daughter branch.
}

See \cite[pg.~17]{GSbook2023} for a discussion of the meaning of ``generically,'' which is rather subtle.
In summary, generic behavior is typical.
%

\item Every EBL bifurcation is a BLIS bifurcation.
That is, our Theorem~\ref{thm:BLIS} implies the results usually obtained by an equivariant  Lyapunov-Schmidt reduction 
followed by an application of the EBL \cite{cicogna, GSS, vander}: 

{\it Assume that a compact Lie group $\Gamma$ acts on a Banach space $X$, that 
$F:I \times X \to X$ is $\Gamma$-equivariant, and that
$\Sigma_d < \Sigma_m \leq \Gamma$.
Let  $x^* \in \fix(\Sigma_m)$ satisfy $F(s^*, x^*) = 0$ and 
let $K = N(D_2 F(s^*, x^*))$.
If 
$$
K\cap  \fix(\Sigma_m) = \{0\} \text{ and } \dim (K\cap  \fix(\Sigma_d)) = 1,
$$
then generically a branch of solutions in $I \times \fix(\Sigma_d)$ bifurcates from the mother branch,
which is
in $I \times \fix(\Sigma_m)$,
at $(s^*, x^*)$.
}
\item 
Our Theorem~\ref{thm:BLIS} is a generalization of the Synchrony Branching Lemma for regular coupled cell networks. See \cite[Theorem 6.3]{GL_SBL}, \cite[Theorem~8.2]{FranciGolubitskyStewart} and \cite[Theorem 18.15]{GSbook2023}.  This theorem was generalized in \cite{soares}.  
We state this theorem, indicating the mother and daughter subspaces:

{\it 
Let
$X$ be the finite dimensional phase space of a network with the admissible map $F: \R \times X \to X$, defined in 
\cite{FranciGolubitskyStewart}. Let $W_m = \Delta$ be  
the space of fully synchronous states, and let $W_d = \Delta_{\bowtie}$ be the synchrony subspace for some balanced coloring $\bowtie$.
Let $x^* \in \Delta$, $s^* \in \R$, and let $K = N(D_2 F(s^*, x^*))$.
If 
$$K \cap \Delta = \{0\} \text{ and } \dim (K \cap \Delta_{\bowtie}) = 1,$$
then generically a unique branch of equilibria with synchrony pattern ${\bowtie}$ bifurcates from $(s^*, x^*)$.
}
\item
There might be many other branches bifurcating from $(s^*, x^*)$.  
When we restrict to branches in $I \times W_d$, however, there are exactly two branches (the mother and daughter branches).
\item
\label{rem:XandWd}
Note that
$$
N(J(s^*)) := N(D_2 F_d(s^*, x^*) ) = N( D_2 F(s^*, x^*)) \cap W_d .
$$
As in Part (1) of this remark, let $K = N( D_2 F(s^*, x^*))$.  One of the hypotheses in Proposition~\ref{prop:motherExists}
is equivalent to $K \cap W_m = \{0\}$. 
Condition (a) of Theorem~\ref{thm:BLIS} implies that $\dim (K \cap W_d) = 1$.
\item
If $W_d$ is finite dimensional, then the two parts of Condition (a) of Theorem~\ref{thm:BLIS} are equivalent.
That is, we only need to check that $N(J(s^*))$ is one-dimensional.
\item
Rather than having one daughter branch that crosses the mother branch at the bifurcation point, 
we often consider 
two daughter branches $C_d^+$ and $C_d^-$, 
born at the bifurcation, 
defined by
$$
C_d^\pm := \{(\varphi(\alpha), b_m(\varphi(\alpha) + \alpha x_0 + \alpha \psi(\alpha) ) : 0 < \pm \alpha  < a \} .
$$   
Note that $0 < \alpha < a$ for the parametric curve $C_d^+$ while $-a < \alpha < 0$ for $C_d^-$.
Neither of these branches include $(s^*, x^*)$, and
$$
C_m \cup C_d = C_m \dot\cup C_d^+ \dot \cup C_d^-.
$$
The MATLAB$\textsuperscript{\textregistered}$ code, described in Section \ref{sec:Implementation} and freely available at \cite{NSS8_companion}, follows the branches $C_d^+$ and $C_d^-$ separately, starting from their birth at $(s^*, x^*)$.
At a pitchfork bifurcation where the two daughter branches are conjugate by a symmetry of the system, 
the code only follows $C_d^+$. 
\end{enumerate}
\end{rem}

\section{Applications of the BLIS}
\label{sec:Examples}
We consider several examples of bifurcations and classify them as BSE, EBL, or BLIS, as indicated in Figure~\ref{VennDiagram}.
\subsection{One-Dimensional Examples}
We consider some applications of BSE, EBL, and BLIS in the case where $X = \R$.
We will write examples of ODES, $\dot x = F(s, x)$, and look for the equilibrium solutions which satisfy
$F(s, x)  = 0$, for functions $F: \R \to \R$.  In such cases, $J(s)$ in the statement of Theorem~\ref{thm:BLIS} 
is a $1 \times 1$ matrix.
\begin{example}
\label{pitchforkInR}
The ODE
$$
\dot x = F(s, x) := s x - x^3
$$
has a pitchfork bifurcation at $(s, x) = (0, 0)$.
In this case, $W_m = \{0\} \subseteq W_d = \R$.  Factoring $F$ shows that the mother branch has the equation $x = 0$, and the daughter branch has the equation $s = x^2$.

The mother branch function is $b_m(s) = 0$ for all $s$,
and $J(s) = s \equiv [s] \in \R^{1 \times 1}$.  The bifurcation point is $s^* = 0$, and $J(0) = 0$, and $J'(0) = 1$.
The system has the symmetry $F(s, -x) = -F(s, x)$, so both $W_m$ and $W_d$ are fixed point subspaces.
The bifurcation is BSE, EBL, and BLIS.
See Figure~\ref{1Dbifs}(a).
\end{example}

\begin{example}
\label{transInR}
The ODE
$$
\dot x = F(s, x) := s x - x^2
$$
has a so-called transcritical bifurcation at $(s, x) = (0, 0)$.
In this case, $W_m = \{0\} \subseteq W_d = \R$.  Factoring $F$ shows that the mother branch has the equation $x = 0$, and the daughter branch has the equation $s = x$.

The mother branch function is $b_m(s) = 0$ for all $s$,
and $J(s) = [s] \in \R^{1 \times 1}$.  The bifurcation point is $s^* = 0$, and $J(0) = [0]$, and $J'(0) = [1]$.
The trivial subspace is $F$-invariant since $F(s, 0) = 0$ for all $s$,
The bifurcation is BSE and BLIS.
See Figure~\ref{1Dbifs}(b).
\end{example}

\begin{example}
\label{trans2InR}
The ODE
$$
\dot x = F(s, x) := s^2 - x^2
$$
has a different kind of transcritical bifurcation at $(s, x) = (0, 0)$.
In this case, the only $F$-invariant subspace is $\R$, so this is not a BLIS bifurcation.
The function $F$ has no symmetry, so this is not an EBL bifurcation.
Factoring $F$ shows that there are two branches, $x = s$ and $s = -x$.  
Arbitrarily choose one of these branches,
$x = b_m(s) := -s$ and define $\tilde F(s, x) = F(s, x - b_m(s)) = s^2 - (x+s)^2 = 2 s x - x^2$.  
The BSE Theorem~\ref{thm:CR1.7} and BLIS Theorem~\ref{thm:BLIS} both apply to $\tilde F$ with $s^* = 0$.  
The relevant calculations are that $D_2 \tilde F(s, x) = 2s - 2x$, so $D_2 \tilde F(0,0) = 0$
and $D_1 D_2F(s, x) = 2$. The critical eigenvector is $x_0 = 1$, so $\spn(\{x_0\}) = \R$.
It is clear from factoring $\tilde F$ that the two branches of $\tilde F(s, x) = 0$ are $x = 2 s$ and $x = 0$, so
$$
\tilde F^{-1}(\{0\}) = \{ (\alpha/2, \alpha) : \alpha \in \R \} \cup \{(s, 0) : s \in \R\}.
$$
Theorems~\ref{thm:CR1.7} and \ref{thm:BLIS} hold with $\phi(\alpha) = \alpha/2$.  

Theorem~1 of  \cite{CR} describes how the original function $F$, which a branch of nontrivial solutions, has a bifurcation.
Thus we say that $F$ has a BSE
bifurcation, as indicated in Figure~\ref{VennDiagram}.  

Although Theorem~\ref{thm:BLIS} applies to the reduced function $\tilde F$, it does not apply directly to $F$.  
Since the original function $F$ has no nested invariant subspaces,
we do not include this example in the BLIS region of Figure~\ref{VennDiagram}.

Note that this example does not occur generically in one-parameter families.  If a small
term $\eps$ is added, then the system becomes $\dot x = F(s, \eps, x) := s^2 + \eps - x^2$,
which has a qualitatively different bifurcation diagram of $x$ vs. $s$ for fixed $\eps > 0$, $\eps = 0$, or $\eps < 0$.
In contrast, the transcritical bifurcation in Example~\ref{transInR} is generic for the class of systems with $F(s, 0) = 0$, which is a natural
constraint for some models.
See Figure~\ref{1Dbifs}(c).
\end{example}

\begin{example}
\label{foldInR}
The ODE
$$
\dot x = F(s, x) := s - x^2
$$
has the equilibrium solutions
$$
F^{-1}(\{0\}) = \{ (\alpha^2, \alpha) : \alpha \in \R \}.
$$
The point $(s,x) = (0,0)$ is a fold point, 
We do not call this a saddle-node bifurcation, 
following our Definition~\ref{def:bifPt}.
Since there is only one solution branch and no branching of solutions, 
neither the  BSE nor the BLIS apply.
See Figure~\ref{1Dbifs}(d).
\end{example}

\begin{figure}
\begin{center}
\begin{tikzpicture}[scale =2.0]
\draw (.58,0) -- ++(-.5,0);
\draw[dashed] (.58,0) -- ++(.5,0);
\draw(.58,0) arc [start angle = 180, end angle = 235, x radius = 1.1, y radius = .6]; 
\draw(.58,0) arc [start angle = 180, end angle = 125, x radius = 1.1, y radius = .6];
\node[label=above : {(a)}] at (.6,-.85) {};
\filldraw (.58,0) circle [radius=.6pt];

\draw  (2.15,0) -- ++(-.5,0);
\draw[dashed]  (2.15,0) -- ++(.5,0);
\draw [dashed] (2.15,0) -- ++(-.5,-.5);
\draw (2.15,0) --++(.5,.5);
\node[label=above : {(b)}] at (2.15,-.85) {};
\filldraw (2.15,0) circle [radius=.6pt];

\draw[dashed]  (3.9,0) -- ++(-.5,-.5);
\draw[dashed]  (3.9,0) -- ++(.5,-.5);
\draw               (3.9,0) -- ++(.5,.5);
\draw               (3.9,.0) -- ++(-.5,.5);
\node[label=above : {(c)}] at (3.9,-.85) {};
\filldraw (3.9,0) circle [radius=.6pt];

\draw (5.1,0) arc [start angle = 180, end angle = 125, x radius = 1.1, y radius = .6];
\draw[dashed] (5.1,0) arc [start angle = 180, end angle = 235, x radius = 1.1, y radius = .6];
\node[label=above : {(d)}] at (5.2,-.85) {};
\filldraw (5.1,0) circle [radius=.6pt];

\draw [dashed] (6.6,0) -- ++(-.5,0);
\draw[dashed]  (6.6,0) -- ++(.5,0);
\draw (6.6,0) -- ++(-.5,-.5);
\draw (6.6,0) --++(.5,.5);
\draw (6.6,0) --++(-.5,.5);
\draw (6.6,0) --++(.5,-.5);
\node[label=above : {(e)}] at (6.6,-.85) {};
\filldraw (6.6,0) circle [radius=.6pt];

\end{tikzpicture}
\end{center}

\caption{
\label{1Dbifs}
The $x$ vs.\ $s$ bifurcation diagrams of the ODEs featured in Examples \ref{pitchforkInR}(a), 
\ref{transInR}(b), \ref{trans2InR}(c), \ref{foldInR}(d), and \ref{nongenericInR}(e).
In each case the dot is at $(s,x) = (0,0)$, and the solid lines represent a branch of stable equilibrium solutions to $\dot x = F(s, x)$.  The dashed lines
represent a branch of unstable equilibrium solutions.
}
\end{figure}
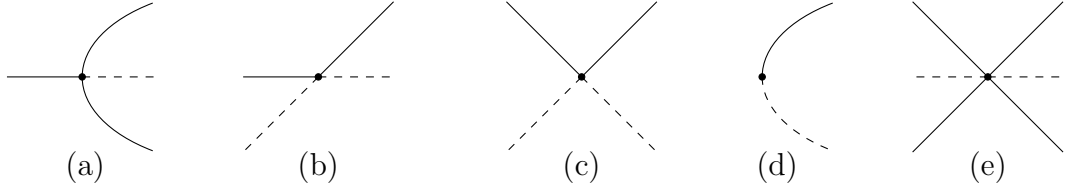

\begin{example}
\label{nongenericInR}
Consider the ODE
$$
\dot x = F(s, x) := s^2 x - x^3 = x(s+x)(s- x).
$$
The mother branch function is $b_m(s) = 0$, and the stability of the trivial branch
is determined by $J(s) = s^2$. 
Condition (a) of Theorem~\ref{thm:BLIS} holds, with  $(s^*, x^*) = (0,0)$, 
$W_m = \{0\}$, $W_d = \R$,
and the critical eigenvector $x_0 = 1$.
However the nondegeneracy Condition (b) of Theorem~\ref{thm:BLIS}  does not hold,
since $J'(s) = 2s$ and thus $J'(0)(1) = 0 \in R(J(0))$.
Similarly, the BSE does not hold for this example, since Condition (d) of Theorem~\ref{thm:CR1.7}
is false.
The function $F$ does have symmetry, but the EBL does not hold.
This nongeneric example is one of the exceptions referred to in 
Remark~\ref{rem:BLISremarks}(1).
The BLIS says that 
generically there are exactly two solution branches in $\R \times X$ that contain $(s^*, x^*)$.
Similar examples of $F: \R \times X \to X$ for which Condition (a) of the BLIS holds, 
but Condition (b) does not hold, can be manufactured for any $X$.
See Figure~\ref{1Dbifs}(e).
\end{example}

\subsection{Network Examples}
We apply the BLIS to systems of networks of $n$ coupled cells defined by $\dot x = F(s, x)$ for $F: \R \times (\R^k)^n \to (\R^k)^n$.  The ODE for each cell with phase space $x_i \in \R^k$ is
\begin{equation}
\label{coupledCells}
{\dot x}_i = F(s, x)_i := f(s,x_i) + H \sum_{j=1}^n M_{i,j} x_j = f(s, x_i) + H (Mx)_i,
\end{equation}
where $f: \R \times \R^k \to \R^k$ describes the internal dynamics of each identical cell, 
$H \in \R^{k\times k}$ is the coupling matrix, 
and $M\in \R^{n\times n}$ is the \emph{weighted adjacency matrix} of the network.
Note that we multiply a real matrix $M\in\R^{n\times n}$ with a matrix $x\in (\R^k)^{n\times 1}$ whose entries come from $\R^k$, resulting in a matrix $Mx\in(\R^k)^{n\times 1}$, which we identify with $(\R^k)^n$. 

In order to understand the BLIS bifurcations, we first
describe the subspaces that are invariant under System~(\ref{coupledCells}) for a fixed $M$ with any $f$ and $H$. 
These invariant subspaces are related to the $M$-invariant \emph{polydiagonal subspaces} of $\R^n$, as described in \cite{NiSS}.
Here we show that
the dynamics on one of these invariant subspaces is also described by System~(\ref{coupledCells}) but with
$d$ coupled cells, where the 
$n\times n$ matrix $M$ is replaced by a smaller $d \times d$ matrix.  
We can think of any matrix $M$ as the adjacency matrix of a weighted digraph with $n$ vertices, 
and then the new $d\times d$
matrix is the weighted adjacency matrix of the 
\emph{weighted quotient digraph} with $d$ vertices.

We use the definition of polydiagonal subspaces found in \cite{NiSS}, which includes both synchrony and anti-synchrony subspaces.
This development builds on the work of \cite{AD2021,GSbook2023}, but the terminology is different.
Here we provide a
characterization of a nontrivial polydiagonal subspace 
as the column space of a matrix with special properties.  This new characterization
lends itself well to numerical computations, and is used 
in our MATLAB$\textsuperscript{\textregistered}$ program \cite{NSS8_companion}.

\begin{definition} 
\label{def:polyBasisMatrix} 
For $0 < d \leq n$, 
a \emph{polydiagonal basis matrix} is  a matrix $B \in \{1,0, -1\}^{n \times d}$
that satisfies the following three conditions:  
\begin{enumerate}
\item $\text{rank}\,(B) = d$;
\item every row of $B$ has at most one non-zero entry;
\item $B^T$ is in reduced row-echelon form.
\end{enumerate}
If, in addition, every row of $B$ has an entry of 1, then $B$ is called a \emph{synchrony basis matrix}.
Otherwise, $B$ is called an \emph{anti-synchrony basis matrix}.
\end{definition}

Thus, a synchrony basis matrix only has elements 0 and 1, and every row has a single 1.
An anti-synchrony basis matrix has a row of all 0's or a row with a single $-1$.

\begin{prop}
A nontrivial polydiagonal subspace of $\R^n$ is the column space of a unique polydiagonal matrix. 
Furthermore, $\col(B)$ is a synchrony/anti-synchrony subspace if and only if $B$ is a synchrony/anti-synchrony basis matrix.
\end{prop}
\begin{rem}
The proof of this proposition is straightforward but tedious.
The connection with the definition of a polydiagonal subspace in \cite{AD2021,NiSS}  is that a polydiagonal subspace and $\col(B)$ can both be defined by equations of the form 
$x_i = x_j$ or $x_i = - x_j$ (including $x_i = 0$ when $i = j$).
Any set of equations defining an anti-synchrony subspace has at least one equation of the form $x_i = - x_j$.
Conditions (1) and  (3) in Definition~\ref{def:polyBasisMatrix} ensure that $\col(B_1) = \col(B_2)$ implies $B_1 = B_2$.

\end{rem}
Note that the trivial subspace $\{0\} \subseteq \R^n$ is an anti-synchrony subspace according to the definition
in \cite{NiSS}, but it is not the column space of any non-empty matrix.

\begin{example}
The matrix $B := \left[ \begin{smallmatrix} 1 & 0 & 1\\ 0 & 1 &0 \end{smallmatrix} \right ]^T$ is the 
synchrony basis matrix for 
$\{(a, b, a) : a,b \in \R 
\} \subseteq \R^3$.
The corresponding tagged partition, as defined in \cite{NiSS}, is $\mathcal P = \{ \{1,3\}, \{2\} \}$ with the empty partial involution.

The matrix $B := \left[ \begin{smallmatrix} 1 & 0 & 0 & -1 \\ 0 &1 & 0 &0 \end{smallmatrix} \right ]^T$ is the 
anti-synchrony basis matrix for 
$\{(a, b, 0, -a) : a,b \in \R 
\} \subseteq \R^4$.
The corresponding tagged partition is $\mathcal P = \{ \{1\} ,\{2\}, \{3\}, \{4\} \}$ with $\{1\}^* = 
\{4\}$ and $\{3\}^* = \{3\}$. Note that $\{2\}$ is not in the domain of the partial involution.
\end{example}

Let $B\in \R^{n\times d}$ be a polydiagonal basis matrix.  
Then $\col(B)$ is 
$M$-invariant if and only if $\col(MB) \subseteq \col(B)$.  See \cite{NSS7}
for an efficient way to find the $M$-invariant synchrony subspaces of $\R^n$. 
Note that the columns of $B$ are orthogonal due to Condition~(2), 
and $B^T B$ is a diagonal $d\times d$ matrix, with $(B^TB)_{\ell,\ell}$ equal to
the number of nonzero components in the $\ell$th column of $B$.  
Condition (1) implies that $(B^TB)_{\ell,\ell} > 0$.
The  Moore-Penrose inverse (pseudoinverse)
$$
B^+ = (B^T B)^{-1} B^T \in \R^{d\times n}
$$
of $B$ is easily computed since $B^TB$ is diagonal and nonsingular.
The pseudoinverse satisfies $B^+B = I_d$ and $B B^+$ is the projection of $\R^n$ onto $\col(B)$.

\begin{example}
\label{Bplus}
Let $B = \left[ \begin{smallmatrix} 1 & 0 & 0 & -1 \\ 0 &1 & 0 &0 \end{smallmatrix} \right ]^T$.
Then $B^T B = \left[ \begin{smallmatrix} 2 & 0  \\ 0 &1 \end{smallmatrix} \right ]$,
so
$$
B^+ =  \begin{bmatrix}2 & 0 \\ 0 &1 \end{bmatrix}^{-1}  \begin{bmatrix} 1 & 0 & 0 & -1 \\ 0 &1 & 0 &0 \end{bmatrix}  =  \begin{bmatrix} 0.5 & 0 & 0 & -0.5 \\ 0 &1 & 0 &0 \end{bmatrix}
$$
is the pseudoinverse of $B$.
\end{example}

We now consider the dynamically invariant subspaces of 
System~(\ref{coupledCells}), which are precisely the $F$-invariant subspaces.
Proposition~4.1 of \cite{NiSS} describes these subspaces.
If $M$ and $H \neq 0$ are fixed, then the 
nontrivial
subspaces of $(\R^k)^n$ 
that are $F$-invariant for all odd $f$ are precisely the subspaces of the form
$$
W_B := \{B y : y \in (\R^k)^d\} \subseteq (\R^k)^n
$$
for some polydiagonal basis matrix
$B \in \R^{n \times d}$ such that $\col(B)$ is $M$-invariant.
The trivial subspace $\{0\} \subseteq (\R^k)^n$
$F$-invariant whenever $f(0) = 0$.
Furthermore, if $f$ is not odd and $B$ is a synchrony basis matrix, then $W_B$ is $F$-invariant.

Note that $By\in (\R^k)^{d\times 1}\equiv(\R^k)^d$ is a product similar to 
$Mx$ in System~(\ref{coupledCells}). If  $k = 1$, then $W_B = \col(B)$. 
If $k > 1$, then $W_B = \R^k \otimes \col(B)$, which is the tensor product of $\R^k$ with the polydiagonal subspace $\col(B)$ \cite{NiSS}.

\begin{example}
Let $B = \left[ \begin{smallmatrix} 1 & 0 & 0 & -1 \\ 0 &1 & 0 &0 \end{smallmatrix} \right ]^T$ and $k = 3$, so that $y = (y_1, y_2) \in (\R^3)^2$. We find that 
$$
By = \left[ \begin{matrix} 1 & 0 \\  0 &1 \\ 0 & 0 \\  -1 & 0 \end{matrix} \right ] 
\left [ \begin{matrix} y_1 \\ y_2 \end{matrix} \right ] 
= \left [ \begin{matrix} y_1 \\  y_2 \\ 0 \\ -y_1 \end{matrix} \right ]
\equiv (y_1,y_2,0,-y_1),
$$
so $W_B = \{ (y_1, y_2, 0, -y_1) : y_1, y_2 \in \R^3\} \subseteq (\R^3)^4$. 
A typical element of $W_B$ is $(y_1, y_2, 0, -y_1)$, which we can abbreviate to $(a, b, 0, -a)$, as in Figure~\ref{quotientWeightedDigraph}.
\end{example}

We now describe how to restrict the System~(\ref{coupledCells}) to an invariant subspace $W_B$.
To simplify the exposition we first consider a network of uncoupled cells with no parameter. 

\begin{prop}
\label{uncoupled}
Let $B \in \R^{n \times d}$ be a polydiagonal basis matrix, $f: \R^k \to \R^k$, and
$F: (\R^k)^n \to (\R^k)^n$ defined by $F(x)_i := f(x_i)$.
If $f$ is odd or $B$ is a synchrony basis matrix, then $F_B: (\R^k)^d \to (\R^k)^d$ defined by $F_B(y) := B^+ F(By)$ satisfies 
\[
F_B(y)_\ell = f(y_\ell).
\] 
\end{prop}

\begin{proof}
Let 
\[
N := \{i \mid (\exists \ell) \text{ } B_{i,\ell} \neq 0 \}
\] 
be the set of indices $i$ such that the $i$th row of $B$ has a nonzero element. 
Define $g: N \to \{1, \ldots, d\}$ and $\sig: N \to \{1, -1\}$ such that
$$
B_{i, \ell} = 
\begin{cases}
\sig(i) & \text{if  } \ell = g(i)\\
0       & \text{otherwise.}
\end{cases}
$$
If  $i \not \in N$, then $(By)_i = 0$.  If $i \in N$, then 
$(By)_i = \sum_{\ell = 1}^n B_{i, \ell} y_\ell = \sig(i) y_{g(i)}$ and 
$F(By)_i = f((By)_i) = f(\sig(i) y_{g(i)})$.

If $f$ is odd, then $F(By)_i = f(0) = 0$ for $i \not \in N$ and 
$F(By)_i = f(\sig(i) y_{g(i)} ) =\sig(i) f(y_{g(i)})$ for $i \in N$.
If $B$ is a synchrony basis matrix, then $\sig(i) = 1$  for all $i \in N = \{1, \ldots, n\}$.
In either case, 
\[
F(By)_i =\begin{cases}
\sig(i) f(y_{g(i)}) & \text{if } i \in N      \\
0                   & \text{if } i \not \in N  .
\end{cases}
\]

Let $h(\ell) := (B^TB)_{\ell, \ell}$ be the number of nonzero elements in the $\ell$th column of $B$. 
Note that $h(\ell)$ is the size of the set $g^{-1}(\ell)$.
The pseudoinverse of $B$ has components
$$
B^+_{\ell, i} = 
\begin{cases}
\frac{\sig(i)}{
h(\ell) } & \text{if  } \ell = g(i)\\
0 & \text{otherwise.}
\end{cases}
$$
Then
$$
F_B(y)_\ell 
= \sum_{i = 1}^n B^+_{\ell, i} F(By)_i 
=  \sum_{i \in g^{-1}(\ell)} \frac{\sig(i)}{h(\ell)} \sig(i) f(y_{g(i)}) 
= h(\ell)  \frac{f(y_\ell)}{h(\ell)} 
= f(y_\ell)
$$
since the sum has $h(\ell)$ nonzero terms, all equal to $f(y_\ell)/h(\ell)$.
\end{proof}

\begin{rem}
Proposition~\ref{uncoupled} is the motivation for the definition of a polydiagonal basis matrix. 
Consider $B = [2 \ 0]^T$, 
which is not a polydiagonal basis matrix 
since $B \not \in \{1,0, -1\}^{n \times d}$
even though 
$\col(B)$ is a polydiagonal subspace which is $F$-invariant for the uncoupled system
in Proposition~\ref{uncoupled} if $f(0) = 0$.
Since $F_{B}(y)_1 = \frac 1 2 f(2 y_1)$ and 
in general $\frac 1 2 f(2 y_1) \not = f(y_1)$,
the single reduced cell $\dot y_1 = F_{B}(y)_1$ may not have the internal dynamics $\dot y_1 = f(y_1)$.
\end{rem}


We now use Proposition~\ref{uncoupled} to write $F_B$ in terms of $y \in \R^k$, including the parameter $s$.

\begin{prop}
\label{linearlyCoupled}
Let $B \in \R^{n \times d}$ be a polydiagonal basis matrix, $\col(B)$ be an $M$-invariant polydiagonal subspace, $H \in \R^{k\times k}$, $f: \R \times \R^k \to \R^k$, and
$F: \R \times (\R^k)^n \to (\R^k)^n$ be defined by
\begin{equation}
\label{Fofx}
F(s, x)_i := f(s, x_i) + H \sum_{j=1}^n M_{i,j} x_j .
\end{equation}
Assume that $f$ is odd or $B$ is a synchrony basis matrix.
The function $F_B: \R \times (\R^k)^d \to (\R^k)^d$ defined by $F_B(s, y) = B^+ F(s, By)$ satisfies
\begin{equation}
\label{FBofw}
 F_B(s, y)_\ell = f(s, y_\ell) + H \sum_{m=1}^d(B^+MB)_{\ell, m} y_m .
\end{equation}
\end{prop}

\begin{proof}
The internal dynamics $f$ is handled in Proposition~\ref{uncoupled}, and the inclusion of the parameter $s$ in both $f$ and $F$ 
has no effect on the result.
Recall that $x = By$ means that $x_j = (By)_j = \sum_{m=1}^d B_{j,m} y_m$.
Thus,
\begin{align*}
F_B(s, y)_\ell 
&= \sum_{i = 1}^n B^+_{\ell,i}  F(s,By)_i \\
&= f(s, y_\ell) + \sum_{i = 1}^n B^+_{\ell,i} \left ( H \sum_{j=1}^n M_{i,j} (By)_j \right ) \\
&= f(s, y_\ell) + H\sum_{i = 1}^n \sum_{j = 1}^n \sum_{m= 1}^d B^+_{\ell,i} M_{i,j} B_{j, m} y_m \\
&= f(s, y_\ell) + H \sum_{m= 1}^d (B^+MB)_{\ell,m} y_m.
\end{align*}
\end{proof}

Note the similarity of Equations~(\ref{Fofx}) and (\ref{FBofw}). 
The internal dynamics $f:\R \times \R^k\to \R^k$ is the same in both equations, 
but $M \in \R^{n\times n}$ 
in Equation~(\ref{Fofx}) is replaced by
$B^+MB \in \R^{d\times d}$ in Equation~(\ref{FBofw}).  
We interpret $M$ as the adjacency matrix of a weighted digraph with $n$ vertices,
and $B^+MB$ is the adjacency matrix of the weighted quotient digraph with $d$ vertices.
See Figure~\ref{quotientWeightedDigraph}.
Note that a graph (which is by definition unweighted, undirected, and has no loops) can have 
a weighted quotient digraph which is weighted, directed, and has loops.

While the proofs of Propositions~\ref{uncoupled} and \ref{linearlyCoupled} are complicated, 
the result is obvious when the
$n$ component equations of $F(s, x)$ are restricted to a polydiagonal subspace $W_B$.  It is clear that
only $d = \dim(W_B)$ equations are essential and the rest hold automatically, 
as demonstrated in the next example.
\begin{example}
\label{quotientExample}
Let
$$
M = \begin{bmatrix}
0&1&1&1\\
1&0&0&-1\\
1&0&0&1\\
1&-1&1&0
\end{bmatrix},
\text{  and }
B = 
 \begin{bmatrix}
1&0\\
0&1\\
0&0\\
-1&0
\end{bmatrix},
\text{  so }
MB = 
 \begin{bmatrix}
-1&1\\
2&0\\
0&0\\
1&-1
\end{bmatrix}
$$
Note that $\col(B) = \col(MB)$, so $W_B$
is $M$-invariant.
System~(\ref{coupledCells}) is $\dot x = F(s, x)$, where
\begin{align*}
F(s, x)_1&= f(s, x_1) + H(x_2 + x_3 + x_4) \\
F(s, x)_2 &= f(s, x_2) + H(x_1 - x_4)\\
F(s, x)_3 &= f(s, x_3) + H(x_1 + x_4)\\
F(s, x)_4 &= f(s, x_4) + H(x_1 - x_2 + x_3).
\end{align*}
Assume $f$ is odd. It follows that $W_B$ is an $F$-invariant subspace, and
the ODE can be restricted to $W_B$ by setting $x_3 = 0$, and $x_4 = -x_1$.
We find that $F(s, x)_4 = -F(s, x)_1$ and $F(s, x)_3 = 0$, so only the first and second components are needed.
Proposition~\ref{linearlyCoupled} formalizes this procedure.  
The pseudoinverse of $B$ is computed in Example~\ref{Bplus},
which yields the adjacency matrix
$$
B^+MB = 
\begin{bmatrix}
-1&1\\
2&0
\end{bmatrix}  
$$
of the weighted quotient digraph.
System~(\ref{coupledCells}) restricted to $W_B$ is $\dot y = F_B(s, y)$, where
\begin{align*}
F_B(s, y)_1&= f(s, y_1) + H(-y_1 + y_2) \\
F_B(s, y)_2 &= f(s, y_2) + H(2 y_1).
\end{align*} 
This example demonstrates that a weighted digraph with a symmetric adjacency matrix
can give rise to a weighted quotient digraph with a non-symmetric adjacency matrix,
as shown in Figure~\ref{quotientWeightedDigraph}.
\end{example}

\begin{figure}
\begin{tabular}{ccc}
\includegraphics[scale=1.4]{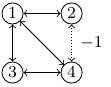}
&
\includegraphics[scale=1.4]{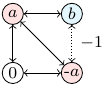}
&
\includegraphics[scale=1.4]{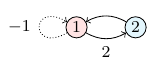}
\\
(a) & (b) & (c)
\end{tabular}

\caption{
\label{quotientWeightedDigraph}
For the matrices $M$ and $B$ defined in Example~\ref{quotientExample}, this figure shows
(a) the weighted digraph whose adjacency matrix is $M$,
(b) a typical element of the $F$-invariant subspace $W_B$, and (c) the weighted quotient digraph with
the adjacency matrix $B^+MB$.
The weights of the arrows in the digraphs are all 1 unless indicated.
}
\end{figure}

We have investigated System~\eqref{coupledCells} for all of the connected graphs with 4 or fewer vertices, with $M$ set to the graph Laplacian matrix. Among these examples only the diamond graph has a bifurcation that is BLIS but neither BSE nor EBL.

\begin{example}
\label{ex:diamond}
Let $G$ be the 4-vertex diamond graph, shown 
in
Figure~\ref{fig:diamondHasse}.  
The graph Laplacian matrix of $G$ is
\begin{equation}
\label{diamondL}
L = 
\begin{bmatrix}
\phantom{-} 2 & -1 & \phantom{-} 0 & -1\\
-1 & \phantom{-} 3 & -1 & -1\\
\phantom{-} 0 & -1 & \phantom{-} 2 & -1 \\
-1 & -1 & -1 & \phantom{-} 3
\end{bmatrix}.
\end{equation}
The lattice of $L$-invariant subspaces are shown in Figure~\ref{fig:diamondHasse}.
\begin{figure}
\includegraphics{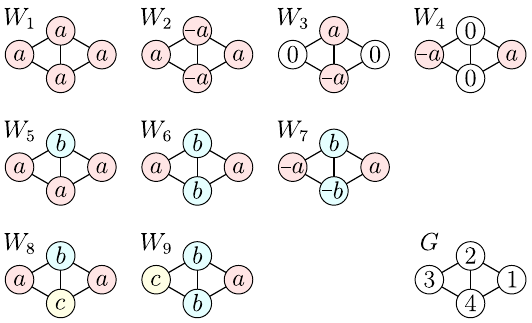}
\quad
\includegraphics{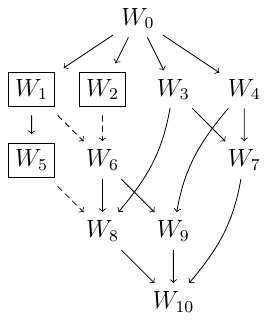}
\caption{
\label{fig:diamondHasse}
The $L$-invariant subspaces for the Laplacian matrix $L$ of the diamond graph $G$.
A typical element of each invariant subspace
is shown on the left, with the exception of $W_0 = \{ 0\} \subseteq \R^4$,
$W_{10} = \R^4$,
 and a reflection of $W_5$.  
The lattice of 
group orbits of
$L$-invariant subspaces
is shown on the right.
The boxes
indicate invariant subspaces that are not fixed point subspaces of the $\aut(G) \times \Z_2$ action.
The dashed arrows connect invariant subspaces with the same point stabilizer.
}
\end{figure}

Assume that the state space of each cell has dimension $k = 1$,
and the internal dynamics is $f(s, x_i) = s x_i + x_i^3$.  We set 
$M = L$, and $H  = [-1] \equiv -1$.
 Note that we identify the $1\times 1$ matrix $[-1]$ with a real number.
Thus,
System~(\ref{coupledCells}) becomes
\begin{equation}
\label{eqn:Fdiamond}
F_i(s, x) = s x_i + x_i^3  - (Lx)_i, \quad i \in \{1, 2, 3, 4\}.
\end{equation}
Thus we have a function $F: \R \times \R^4 \to \R^4$ whose component functions are
Equation~(\ref{eqn:Fdiamond}).
We search for solutions to $F(s, x) = 0$, and we want to understand the bifurcations.
As proved in \cite{NiSS}, the $F$-invariant subspaces are the $L$-invariant subspaces.
These are
shown in 
Figure~\ref{fig:diamondHasse} .

The eigenvalues of $L$ are $\lam = 0$ with eigenvector $(1,1,1,1)$, $\lam = 2$ with eigenvector $(1,0,-1,0)$,
and $\lam = 4$ with an eigenspace spanned by $(1,-1,1,-1)$ and $(0,1,0,-1)$.

The vector field $F$ is $\aut(G) \times \Z_2$-equivariant, where $\aut(G) = \langle (1\, 3), (2 \, 4) \rangle \cong \Z_2 \times \Z_2$ is the symmetry group of the graph $G$.  The extra $\Z_2$
symmetry holds because $F(s, -x) = - F(s, x)$, that is, $F$ is odd.  
The vector field $F$ describes a coupled cell network which has more structure than symmetry alone.  
For example, the multiplicity 2 eigenvalue $\lam = 4$ is not caused by symmetry 
since all of the irreducible representations of $\Z_2 \times \Z_2 \times \Z_2$ are one-dimensional 
\cite{NSS3}.

As described in Section~\ref{sec:Implementation}, 
we have written a MATLAB$\textsuperscript{\textregistered}$ program that computes the bifurcation diagrams of a large class of 
coupled cell networks systems. 
We solve Equation~\eqref{eqn:Fdiamond} on the diamond graph 
with the results shown in Figure~\ref{fig:diamondBifDiagram}.
Our program computes one element in each group orbit of branches that is connected to the trivial branch,
by recursively following daughter branches of BLIS bifurcations. 
The bifurcation points are indicated by the circles.  
If branches cross in this view but there is no circle, 
then the branches do not cross in $\R \times \R^4$ and there is not a bifurcation. 
We focus on the three highlighted
bifurcation points of Figure~\ref{fig:diamondBifDiagram} in the next three examples.
\end{example}

\begin{figure}
\begin{tikzpicture}
\node at (0,0) {\includegraphics[scale=.5]{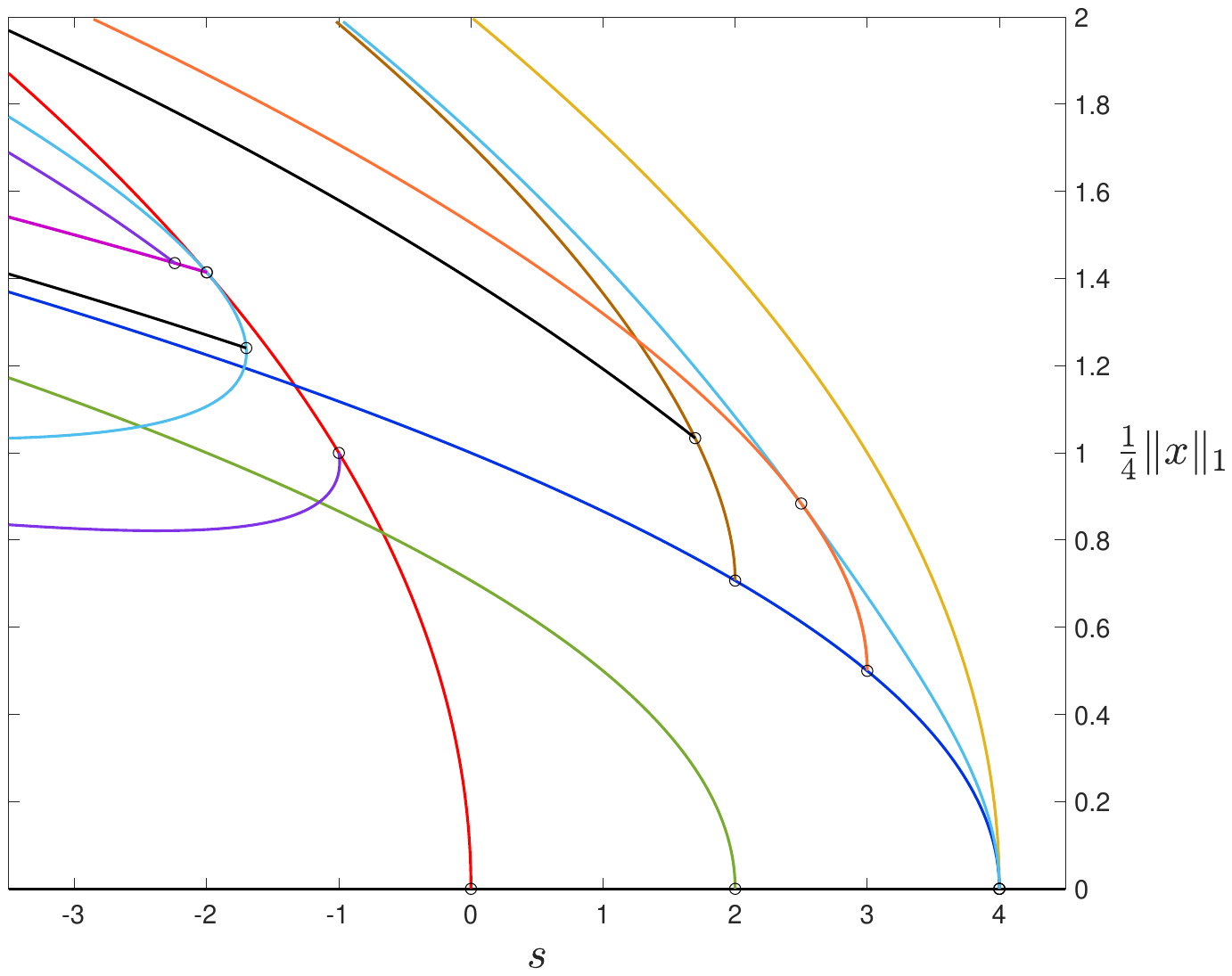}};
\node at (-5.43,4.3) {\small$W_{10}$};
\node at (-5.5,3.88) {\small$W_{1}$};
\node at (-5.5,3.47) {\small$W_{5}$};
\node at (-5.5,3.1) {\small$W_{9}$};
\node at (-5.5,2.51) {\small$W_{6}$};
\node at (-5.43,2.07) {\small$W_{10}$};
\node at (-5.5,1.65) {\small$W_{3}$};
\node at (-5.5,0.97) {\small$W_{4}$};
\node at (-5.5,.39) {\small$W_{5}$};
\node at (-5.5,-.4) {\small$W_{9}$};
\node at (-5.5,-3.9) {\small$W_0$};
\node at (-4.4,4.6) {\small$W_8$};
\node at (-2.3,4.6) {\small$W_7$};
\node at (-1.7,4.6) {\small$W_5$};
\node at (-0.6,4.6) {\small$W_2$};
\node at (-2.6,2.05) {\tiny\ref{W1toW5andW6}};
\node at (2.9,0.0) {\tiny\ref{W5toW8}};
\node at (3.9,-3.7) {\tiny\ref{W0s4}};
\end{tikzpicture}

\caption{
\label{fig:diamondBifDiagram}
The complete bifurcation diagram of Equation~(\ref{eqn:Fdiamond}).
The invariant subspace of each branch is indicated.  The invariant subspaces are described in
Figure~\ref{fig:diamondHasse}.
This figure shows the bifurcation points that spawn BLIS branches that are neither BSE nor EBL branches,
 at $s = -2$ (Example~\ref{W1toW5andW6}), $s = 2.5$ (Example~\ref{W5toW8}), and $s = 4$ (Example~\ref{W0s4}).
}
\end{figure}

\begin{example}
\label{W1toW5andW6}
\begin{figure}
\begin{tikzpicture}
\node at (0,0) {\includegraphics[scale=.5]{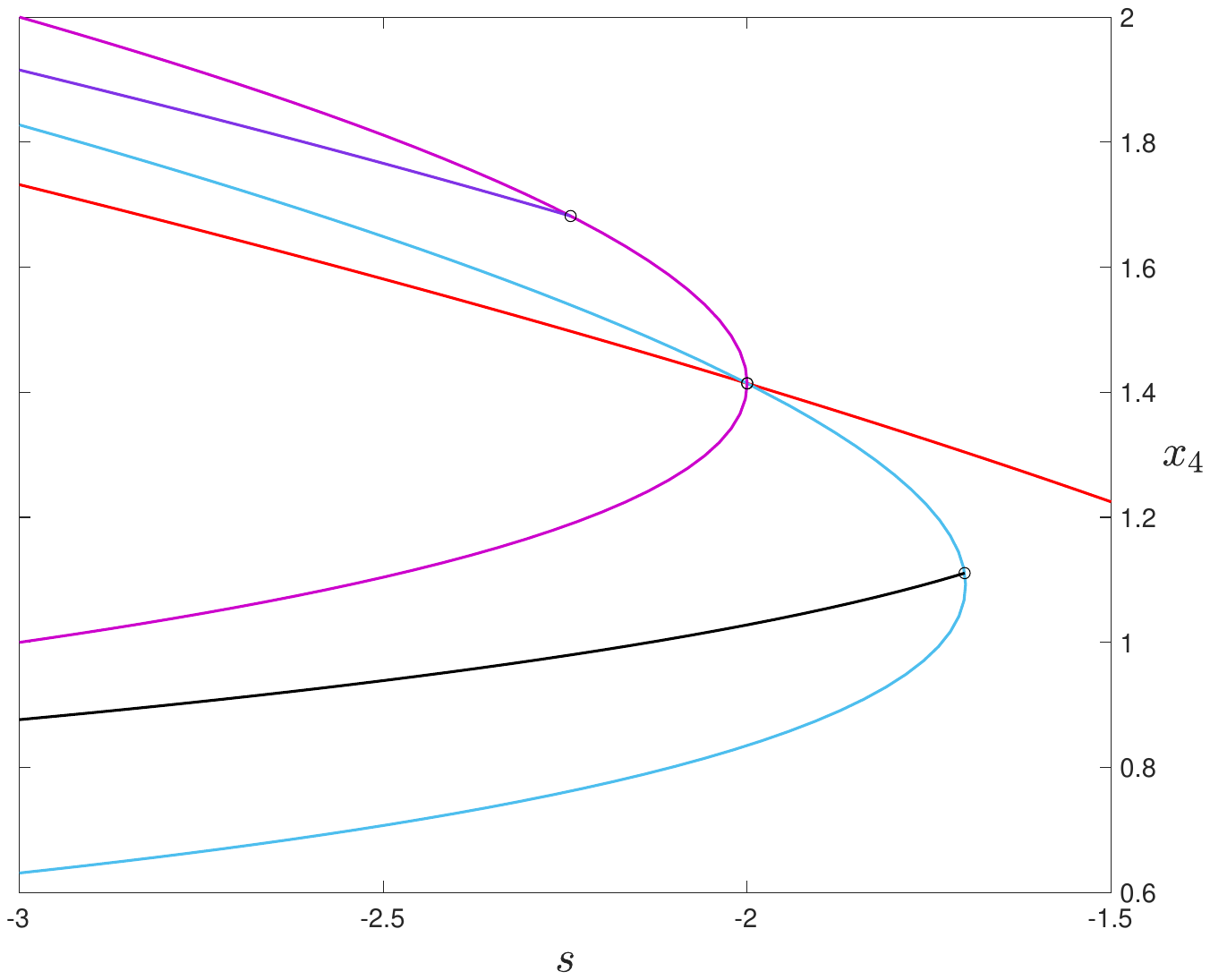}};
\node at (-5.4,4.45) {\small$W_{6}$};
\node at (-5.4,3.9) {\small$W_{9}$};
\node at (-5.4,3.4) {\small$W_{5}$};
\node at (-5.4,2.8) {\small$W_{1}$};
\node at (-5.4,-1.6) {\small$W_{6}$};
\node at (-5.33,-2.3) {\small$W_{10}$};
\node at (-5.4,-3.8) {\small$W_5$};
\end{tikzpicture}
\caption{
\label{fig:W1BLIS}
BLIS bifurcations $W_1 \to W_5$ and $W_1 \to W_6$ at $s = -2$, 
listed as Example~\ref{W1toW5andW6}(a) in Figure~\ref{VennDiagram}.  
There are also two pitchfork bifurcations shown that are BSE, EBL, and BLIS: $W_6 \to W_9$ and $W_5 \to W_{10}$,
which are listed as Example~\ref{W1toW5andW6}(b) in Figure~\ref{VennDiagram}.
Only one of the two branches is plotted in each case to avoid clutter. 
Note that this figure plots a different function of $x \in \R^4$ vs. $s$ than the one plotted
in Figure~\ref{fig:diamondBifDiagram}.  
Plotting $x_4$ shows that only one of the $W_6$ branches has a bifurcation.
See Example~\ref{W1toW5andW6}.  
}
\end{figure}
 Continuing the example of the diamond graph, the solution branch in $W_1$ that bifurcates from the origin at $s = 0$ undergoes two secondary bifurcations,
 at $s = -1$ and $s = -2$, as seen in Figure~\ref{fig:diamondBifDiagram}.  In this example,  we focus on the 
 secondary bifurcation at $s = -2$, and the tertiary bifurcations of these daughters, 
 as seen in the detailed bifurcation diagram
 in Figure~\ref{fig:W1BLIS}.  
 The bifurcation of the $W_1$ branch at $s = -2$ is 
 a BLIS bifurcation 
 that is neither an EBL nor a BSE bifurcation,
 since $W_1$ is not a fixed-point subspace, and
 the critical eigenspace is not one-dimensional.
 This bifurcation is also described by~\cite[Theorem 8.2]{FranciGolubitskyStewart}.
It is indicated in Figure~\ref{VennDiagram} as \ref{W1toW5andW6}(a).  
 The daughters are two secondary branches in $W_5$ and $W_6$, and each of these branches
 undergoes a tertiary bifurcation which is a standard pitchfork bifurcation
 described by the BSE, EBL and BLIS.  These pitchfork bifurcations are indicated in Figure~\ref{VennDiagram} as \ref{W1toW5andW6}(b).

 The synchrony basis matrix for $W_1$ is $B_1 := [1 \, 1\, 1 \, 1]^T$ with pseudoinverse 
 $B_1^+ = \frac 1 4 [1 \, 1\, 1 \, 1]$.  Thus $B_1^+LB_1 = 0$ for the Laplacian matrix (\ref{diamondL}), 
 and 
 Proposition~\ref{linearlyCoupled} 
  says that $F(s, x) = 0$ restricted to $W_1$ is
 $$
 F_{B_1}(s, y) = s y + y^3 = 0.
 $$
 This has three solutions, $y = 0$ and $y = \pm\sqrt{-s}$. 
 The solution $y=0$ is in $W_0$.  
The latter solutions are in $W_1$ but not in $W_0$.
For $y = \sqrt{-s}$, the mother branch function 
 $b_m: \R \to \R^4$ is $b_m(s) = \sqrt{-s}(1,1,1,1)$.
 We find that $D_2 F_{B_1}(s, y) = s + 3 y^2$, and thus $D_2 F_{B_1}(s, b_m(s)) = -2s$.
 This eigenvalue of this matrix does not change sign, and Proposition~\ref{prop:motherExists} applies.
 The eigenvalues of $D_2F(s, b_m(s))$ are $-2s$, $1-s$, and $2-s$ (with multiplicity 2).  We focus on the two BLIS bifurcations at $s = -2$ which
have daughter branches in $W_5$ and $W_6$, respectively.

The synchrony basis matrix of $W_5$ is $B_5 := \left[ \begin{smallmatrix} 1 & 0 & 1 &1 \\ 0& 1&0& 0\end{smallmatrix} \right]^T$, and $B_5^+ L B_5= \left[ \begin{smallmatrix} 1 & -1  \\ -3& 3 \end{smallmatrix} \right]$.
Note that $B_5^+ L B_5$ is the Laplacian matrix of the quotient digraph, which has an arrow with weight 1 from cell 2 to cell 1, and an arrow with weight 3 from cell 1 to cell 2.
By Proposition~\ref{linearlyCoupled}, the system $F(s, x) = 0$ restricted to $W_5$ is
\begin{equation}
\label{W5system}
F_{B_5}(s,y) = (sy_1 + y_1^3 - (y_1-y_2) , s y_2 + y_2^3 -3 (y_2-y_1) )= 0.
\end{equation}
The Jacobian matrix of $F_{B_5}$ is
$$
D_2 F_{B_5}(s, y) =  \begin{bmatrix} s + 3 y_1^2 - 1 & 1 \\ 3 & s + 3y_2^2 - 3 \end{bmatrix}, 
$$
and the Jacobian matrix evaluated at the mother branch is
$$
J(s) := D_2 F_{B_5}(s, b_m(s)) = \begin{bmatrix} -2s - 1 & 1 \\ 3 & -2s  - 3\end{bmatrix}.
$$
The eigenvalues of $J(s)$ are $-2s$ and $-2(s+2)$.  Thus there is a bifurcation at $s^* = 2$ and the two components of Theorem~\ref{thm:BLIS} are
$$
J(-2) = \begin{bmatrix} 3 & 1 \\ 3 & 1\end{bmatrix}, \quad J'(-2) = \begin{bmatrix} -2 & 0 \\ 0 & -2\end{bmatrix}.
$$
We see that $N(J(-2)) = \spn (\{ (1, -3)\})$ is one-dimensional, and the $J'(-2)(1, -3) = (-2, 6)$ is not in $R(J(-2)) = \spn (\{ (1, 1)\})$, 
so the hypotheses of Theorem~\ref{thm:BLIS} hold.
There is exactly one daughter branch in $W_5$ bifurcating from the mother branch at $s = -2$. 

Similarly, the synchrony basis matrix of $W_6$ is $B_6 := \left[ \begin{smallmatrix} 1 & 0 & 1 &0 \\ 0& 1&0& 1\end{smallmatrix} \right]^T$.
The system $F(s, x) = 0$ restricted to $W_6$ is
\begin{equation}
\label{W6system}
F_{B_6}(s,y) = (sy_1 + y_1^3 - 2(y_1-y_2) , s y_2 + y_2^3 -2 (y_2-y_1) ) = 0.
\end{equation}
The Jacobian matrix of $F_{B_5}$ is
$$
D_2 F_{B_6}(s, y) =  \begin{bmatrix} s + 3 y_1^2 - 2 & 2 \\ 2 & s + 3y_2^2 - 2 \end{bmatrix}, 
$$
and the Jacobian matrix evaluated at the mother branch is
$$
J(s) := D_2 F_{B_6}(s, b_m(s)) = \begin{bmatrix} -2s - 2 & 2 \\ 2 & -2s  -2\end{bmatrix}.
$$
The eigenvalues of $J(s)$ are $-2s$ and $-2(s+2)$.  Thus there is a bifurcation at $s^* = 2$ and the two components of Theorem~\ref{thm:BLIS} are
$$
J(-2) = \begin{bmatrix} 2 & 2 \\ 2 & 2\end{bmatrix}, \quad J'(-2) = \begin{bmatrix} -2 & 0 \\ 0 & -2\end{bmatrix}.
$$
We see that $N(J(-2)) = \spn( \{ (1, -1)\})$ is one-dimensional, and the $J'(-2)(1, -1) = (-2, 2)$ is not in $R(J(-2)) = \spn( \{ (1, 1)\})$, 
so the hypotheses of Theorem~\ref{thm:BLIS} hold.
There is exactly one daughter branch in $W_6$ bifurcating from the mother branch at $s = -2$. 

Figure~\ref{fig:W1BLIS} shows these two bifurcations.  Note that the restricted system (\ref{W6system}) has a symmetry $(y_1, y_2) \mapsto (y_2, y_1)$
and consequently the bifurcation $W_1 \to W_6$ is a pitchfork bifurcation within the restricted system
(\ref{W6system}).  Within the full system in $\R \times \R^4$
there is no such symmetry, so our MATLAB$\textsuperscript{\textregistered}$ code follows both branches.
The lower $W_6$ branch in Figure~\ref{fig:W1BLIS} undergoes a bifurcation at $s \approx -2.4$, whereas the upper
branch has no bifurcation.  There is apparently no symmetry of the restricted system (\ref{W5system}), and the bifurcation $W_1 \to W_5$ is transcritical as expected.

%
\end{example}

\begin{example}
\label{W0s4}
The trivial branch in System~\eqref{eqn:Fdiamond} for the diamond graph has a double 0 eigenvalue at $s = 4$.
There is not a BSE at $(s, x) = (4, 0)$ in Figure~\ref{fig:diamondBifDiagram}.
The irreducible representations of $\aut(G) \times \Z_2$ are all one-dimensional, 
so there is not an EBL bifurcation at that point.  
However, there are 3 BLIS branches that bifurcate from this point.
The critical eigenspace of $D_2F(4,0)$ is the eigenspace of $L$ with eigenvalue $\lam = 4$,
 which is $E := \spn(\{(1,-1,1,-1), (0, 1, 0, -1)\})$.  
 We can use a shortcut to avoid calculations like those in Example~\ref{W1toW5andW6},
 following Remark~\ref{rem:BLISremarks} (1) and (5).
 For each of the invariant subspaces $W_i$, 
 there is a BLIS bifurcation from $W_0$ to $W_i$ if the intersection $E \cap W_i$ is one-dimensional.
Consulting Figure~\ref{fig:diamondHasse}, we see that $E \cap W_2 = W_2$,  $E \cap W_3 = W_3$, 
and $E \cap W_5 = \spn(\{(1,-3,1,1)\})$ are all one-dimensional, and $E\cap W_1 = E \cap W_4 = \{0\}$.  
It is the case that $E\cap W_6 = W_2$ is one-dimensional, so the BLIS theorem says that 
there is a daughter branch in $W_6$, but this branch is also in $W_2$.  
Thus, we do not need to consider the invariant subspaces 
that contain $W_2$, $W_3$, and $W_5$ and we have found all the BLIS bifurcations: $W_0 \to W_2$, 
$W_0 \to W_3$, $W_0 \to W_5$.  
These three bifurcating branches are seen in Figure~\ref{fig:diamondBifDiagram}.
\end{example}

\begin{example}
\label{W5toW8}
The BLIS bifurcation of the mother branch $W_5$ to the daughter branch $W_8$ at $s = 2.5$ is a BSE, but not an EBL bifurcation.  
For this bifurcation, we will not verify Condition~(b) in Theorem~\ref{thm:BLIS}
which is generically true, but focus on Condition~(a).  
The system $F(s, x) = 0$ restricted to $W_5$ was shown in Equation~(\ref{W5system}).
We cannot easily solve that system to find the mother branch function $b_m(s)$ in closed form, 
but our MATLAB$\textsuperscript{\textregistered}$ program finds numerical solutions.  The program also gives a numerical approximation 
of the bifurcation point, and
we can verify by hand that $F_{B_5}(\frac 5 2, (\frac 1 {\sqrt{2}}, - \sqrt{2}) ) = (0,0)$.
Furthermore, $D_2F_{B_5}(\frac 5 2, (\frac{1}{\sqrt{2}}, - \sqrt{2}) )$ is nonsingular, 
so the hypotheses of Proposition~\ref{prop:motherExists} hold.
The system $F(s, x) = 0$ restricted to $W_8$ is $F_{B_8}(s, y) = 0$, where
$$
F_{B_8}(s,y) = (sy_1 + y_1^3 - (2y_1-y_2-y_3) , s y_2 + y_2^3 -(3y_2-2y_1-y_3), sy_3 + y_3^3 - (3y_3 -2y_1-y_2)).
$$
The Jacobian matrix of $F_{B_8}$ is
$$
D_2 F_{B_8}(s, y) =  \begin{bmatrix} s + 3 y_1^2 - 2 & 1 & 1 \\ 2 & s + 3y_2^2 - 3 &1 \\ 2 & 1 & s + 3 y_3^2 -3 \end{bmatrix}.
$$
While we cannot find the mother branch function $b_m$, or $J(s) = D_2 F_{B_8}(s, b_m(s))$, 
the Jacobian matrix evaluated at the bifurcation point 
is
$$
\textstyle J(\frac 5 2 ) = D_2 F_{B_8}(\frac 5 2,(\frac 1 {\sqrt{2}}, -\sqrt{2}, \frac 1 {\sqrt{2}}) = 
\begin{bmatrix} 2 & 1 & 1 \\ 2 & \frac{11} 2 &1 \\ 2 & 1 & 1 \end{bmatrix}.
$$
This matrix has a simple zero eigenvalue with an eigenvector of $(1, 0, -2)$.  Thus, Condition~(a) of
Theorem~\ref{thm:BLIS} is satisfied.  While we cannot prove that Condition~(b) in the theorem holds, 
it is true generically, and our MATLAB$\textsuperscript{\textregistered}$ program does not check it.  
The program follows the daughter branches in $W_8$, suggesting that Condition~(b) is true.
This BLIS bifurcation is also a BSE, since the $4 \times 4$ Jacobian matrix evaluated at the bifurcation point has
a simple 0 eigenvalue.  There is no symmetry, so this is not an EBL bifurcation.
\end{example}

\subsection{PDE Examples}
The BLIS applies to find primary and secondary bifurcations for many PDE. 
If all of the hypotheses of the BSE theorem are satisfied except that the critical
eigenspace has dimension greater than one, then one can look for one-dimensional invariant subspaces of the critical eigenspace and apply the BLIS there.  In this section we provide examples from the semilinear elliptic BVP investigated in 
\cite{NSS5,NS}.  
We consider primary bifurcations where the mother branch is the trivial branch defined by $b_m(s)=0$.
In these cases the existence of a daughter branch is proved.  
For secondary bifurcations, the mother branch function is not known explicitly and the bifurcation is 
observed numerically rather than proven to exist.
The existence of nested invariant subspaces can be proved, and suggests a robust numerical algorithm for
following secondary branches.

Consider the PDE
\begin{equation}
\label{PDE}
\Delta u + su + N(u) = 0
\end{equation}
for $u: \Omega \to \R$ with $0\,$-Dirichlet boundary conditions,
where $\Omega$
is a region in $\R^n$,
and $N:\R \to \R$ is a nonlinearity which satisfies $N(0) = N'(0) = 0$.
We seek zeros of the function 
$F:\R \times H \to H$ defined by 
$$F(s, u) = u + \Delta^{-1}(s u + N(u)), 
$$
where $H$ is the Sobolev space $H_0^{1,2}(\Omega)$.
We usually choose the subcritical, superlinear nonlinearity $N(u) = u^3$, as in \cite{NSS5, NS}.
In this case and others, regularity theory \cite{GT} gives that a zero of $F$ is twice 
differentiable and hence a classical solution to the PDE.

The eigenvalue equation associated with PDE~\eqref{PDE} is $\Delta u + \lambda u =0$, 
on the same region with $0\,$-Dirichlet boundary conditions.  
It is well-known \cite{gilbarg2001elliptic} that the eigenvalues are
real and satisfy $0<\lam_1 < \lam_2 \leq \lam_3 \cdots \to \infty$ .
All BSE, EBL, and BLIS bifurcations from the trivial solution are described by Theorem~\ref{thm:BLIS}
with $W_m = \{0\}$. 
For the current function $F$, 
the operator $J(s) = D_2F(s,0)$ defined in Theorem~\ref{thm:BLIS} 
has eigenvalues $\{1-s/\lam_i \mid i \in \N\}$.
We expect that the branch of trivial solutions has a bifurcation of some sort at $(s, u) = (\lam_i, 0)$.
Suppose we find an invariant subspace $W_d \subseteq H$ such that 
$N(J(\lambda_i)) = \spn(\{x_0\})$, in the notation of Theorem~\ref{thm:BLIS}.
We see that all of the conditions of the Theorem hold.  The second part of Condition (a) holds because $H$ is a Hilbert space and 
$\Delta^{-1}$ is self-adjoint. Thus, the range of $J(\lambda_i)$ is the orthogonal complement of the null space,
which has codimension one.
Condition~(b) holds because $J'(s) = \Delta^{-1}$, and therefore
$J'(\lambda_i) (x_0) = x_0/\lambda_i  \not \in R(J(\lambda_i))$.

\begin{example}
\label{transInPDE}
Consider PDE~\eqref{PDE}.
Since $\lam_1$ is simple, 
the BSE and the BLIS apply at the point $(s, u) = (\lam_1, 0)$.
There is a unique branch of nontrivial solutions that bifurcates from this point.
An explicit example is the PDE
$$
\Delta u + s u + u^2  = 0,
$$
for which there is a transcritical bifurcation, and
for which the EBL does not apply.  Note the similarity to Example~\ref{transInR}.
\end{example}

\begin{example}
\label{pitchforkInPDE}
Consider PDE \eqref{PDE} with $N$ odd.  Then $F(s, -u) = -F(s, u)$, and the bifurcation at
$(s, u) = (\lam_1, 0)$ is a pitchfork bifurcation.
In this case, the BSE, EBL, and BLIS all apply.
An explicit example is the PDE
$$
\Delta u + s u + u^3  = 0.
$$
Note the similarity to Example~\ref{pitchforkInR}.
\end{example}

\begin{example}
\label{lam12square}
Consider PDE \eqref{PDE} with the square domain, $\Omega =(0,1)^2$.
The eigenvalues and eigenfunctions of the Laplacian can be explicitly computed as
$$\psi_{n,m}(x, y) = \sin(n \pi x) \sin(m \pi y), \quad 
\lam_{n,m} = (n^2 + m^2) \pi^2,$$
with $n, m \in \N$.
Note that $0 < \lam_{1,1} < \lam_{1,2} = \lam_{2,1}$.
The EBL applies to the bifurcation at $(s, u) = (\lam_{1,2}, 0)$, but BSE does not apply.
The group $\aut (\Omega)$ is the symmetry of the square, and that group acts on $H$.
The fixed point subspaces $\{u \in H: u(x,y) = u(y,x)\} $ and 
$\{u \in H: u(x,y) = u(1-x,y)\}$ are $F$-invariant and the BLIS applies.
\end{example}

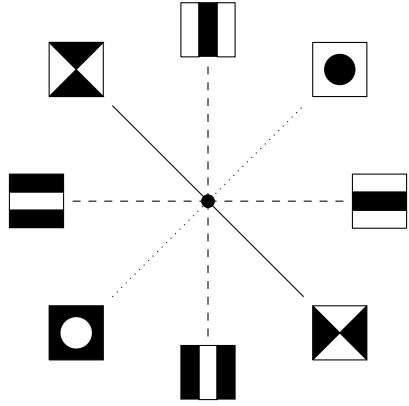
\begin{figure}
\begin{center}
\begin{tikzpicture}[scale =1.2]

\newcommand {\oo} {.44}

\draw[mark=*] plot coordinates{(1.5,1.5)};

\draw[fill=black] (1.2,-.7)--(1.2,-.1)--(1.4,-.1)--(1.4,-.7);
\draw[-] (1.4,-.7)--(1.4,-.1)--(1.6,-.1)--(1.6,-.7)--(1.4,-.7);
\draw[fill=black] (1.6,-.7)--(1.6,-.1)--(1.8,-.1)--(1.8,-.7);

\draw[-] (1.2,3.1)--(1.2,3.7)--(1.4,3.7)--(1.4,3.1)--(1.2,3.1);
\draw[fill=black] (1.4,3.1)--(1.4,3.7)--(1.6,3.7)--(1.6,3.1);
\draw[-] (1.6,3.1)--(1.6,3.7)--(1.8,3.7)--(1.8,3.1)--(1.6,3.1);

\draw[fill=black] (-.7,1.2)--(-.7,1.4)--(-.1,1.4)--(-.1,1.2);
\draw[-] (-.7,1.4)--(-.7,1.6)--(-.1,1.6)--(-.1,1.4)--(-.7,1.4);
\draw[fill=black] (-.7,1.6)--(-.7,1.8)--(-.1,1.8)--(-.1,1.6);

\draw[-] (3.7,1.2)--(3.7,1.4)--(3.1,1.4)--(3.1,1.2)--(3.7,1.2);
\draw[fill=black] (3.7,1.4)--(3.7,1.6)--(3.1,1.6)--(3.1,1.4)--(3.7,1.4);
\draw[-] (3.7,1.6)--(3.7,1.8)--(3.1,1.8)--(3.1,1.6)--(3.7,1.6);

\draw[-] (-.7+\oo,3.1-\oo)--(-.7+\oo,3.7-\oo)--(-.1+\oo,3.7-\oo)--(-.1+\oo,3.1-\oo)--(-.7+\oo,3.1-\oo);
\draw[fill=black] (-.7+\oo,3.1-\oo)--(-.4+\oo,3.4-\oo)--(-.1+\oo,3.1-\oo)--(-.7+\oo,3.1-\oo);
\draw[fill=black] (-.7+\oo,3.7-\oo)--(-.4+\oo,3.4-\oo)--(-.1+\oo,3.7-\oo)--(-.7+\oo,3.7-\oo);

\draw[-] (3.1-\oo,-.7+\oo)--(3.1-\oo,-.1+\oo)--(3.7-\oo,-.1+\oo)--(3.7-\oo,-.7+\oo)--(3.1-\oo,-.7+\oo);
\draw[fill=black] (3.1-\oo,-.7+\oo)--(3.4-\oo,-.4+\oo)--(3.1-\oo,-.1+\oo);
\draw[fill=black] (3.7-\oo,-.7+\oo)--(3.4-\oo,-.4+\oo)--(3.7-\oo,-.1+\oo);

\draw[-] (3.1-\oo,3.1-\oo)--(3.1-\oo,3.7-\oo)--(3.7-\oo,3.7-\oo)--(3.7-\oo,3.1-\oo)--(3.1-\oo,3.1-\oo);
\node at (3.4-\oo,3.4-\oo)   [circle,minimum size=.3cm,fill,draw,thick] {};

\draw[fill=black] (-.7+\oo,-.7+\oo)--(-.1+\oo,-.7+\oo)--(-.1+\oo,-.1+\oo)--(-.7+\oo,-.1+\oo)--(-.7+\oo,-.7+\oo);
\node at (-.4+\oo,-.4+\oo)   [circle,minimum size=.3cm,fill, color=white,draw,thick] {};

\draw [dotted] (.44,.44) -- (2.56,2.56 );
\draw [-] (2.56,.44 ) -- (.44,2.56 );
\draw [dashed] (1.5,0 ) -- (1.5,1.5 );
\draw [dashed] (1.5,1.5 ) -- (1.5,3 );
\draw [dashed] (3,1.5 ) -- (1.5,1.5 );
\draw [dashed] (0,1.5 ) -- (1.5,1.5 );

\end{tikzpicture}
\end{center}

\caption{
\label{psi13_bifs}
The critical eigenspace $V$ at the bifurcation point $(s^*,0)$ in Example~\ref{exa:lambda13}, where $s^*=10\pi^2$. 
The diagonal lines are intersections of two fixed point subspaces with $V$ and lead to EBL bifurcations. 
The horizontal dashed line is the intersection of $W_d$ with $V$ and leads to a BLIS bifurcation. 
The vertical dashed line leads to another similar BLIS bifurcation. 
The three line types correspond to the three group orbits of bifurcating solutions.
}
\end{figure}

\begin{example}
\label{exa:lambda13}
Consider bifurcation point $(s, u) = (\lam_{1,3}, 0)$ for 
PDE \eqref{PDE} with the square domain, $\Omega =(0,1)^2$, with eigenvalues and eigenfunction
listed in Example~\ref{lam12square}.
The eigenvalue $\lam_{1,3} = \lam_{3,1}$
is not simple so BSE does not apply.
The EBL predicts two group orbits of daughter branches, and the BLIS predicts these two  as well as
a third group orbit of bifurcating solutions.
These are shown in Figure~\ref{VennDiagram} as \ref{exa:lambda13}(a) and (b), respectively.

The bifurcation of the trivial solution at $s = \lam_{1,3}$ has a 2-dimensional critical eigenspace
$$
V = N(D_2F(\lam_{1,3}, 0)) = \spn(\{\psi_{1,3}, \psi_{3,1} \}).
$$
Figure~\ref{psi13_bifs} shows the eigenfunctions in $V$ that give rise to 8 solution branches, in three group orbits, that bifurcate from the trivial solution at $s = \lam_{3,1} = 10 \pi^2$. 
We count this as 8 solution branches following Remark~\ref{rem:BLISremarks}(7). 

The function $F$ is $\mathbb D_4 \times \Z_2$-equivariant, where the symmetry group $\mathbb D_4$ of the square acts on $H$ in the natural way and $\Z_2$ acts as $u \mapsto -u$. 
The $\mathbb D_4 \times \Z_2$ action on $V$, modulo the kernel of the action, is isomorphic to $\mathbb D_2$, and is generated by 
$(\psi_{1,3}, \psi_{3,1}) \mapsto (\psi_{3,1}, \psi_{1,3})$ and $(\psi_{1,3}, \psi_{3,1}) \mapsto -(\psi_{1,3}, \psi_{3,1})$.
The EBL says that  
$\spn(\{\psi_{1,3} + \psi_{3,1}\})$ and 
$\spn(\{\psi_{1,3} - \psi_{3,1}\})$, the diagonals in Figure~\ref{psi13_bifs}, give rise to branches in two fixed-point subspaces in $H$. The bifurcating solutions are either even or odd in each of the reflection lines of the square.

The BLIS shows that there is a bifurcating branch in $\R \times W_d$, where
\begin{align*}
W_d := \{u \in H \mid \  & u(x, 1-y) =  u(x,y)  \text{ for all } (x,y)\in (0,1)\times(0,1) \text{ and } \\
                         & u(x, \textstyle \frac 2 3 - y) = -u(x,y)  \text{ for all } (x,y)\in (0,1)\times(0,\textstyle\frac 2 3) \}.
\end{align*}
The intersection of $W_d$ with $V$ is $\spn(\{ \psi_{1,3}\} )$, depicted by the horizontal dashed line in Figure~\ref{psi13_bifs}. 
The stability of the trivial branch  is determined by $J(s) = D_2F_d(s, 0)$, which has a simple eigenvalue equal to $1-s/\lambda_{1,3}$ with the critical eigenvector $\psi_{1,3}$.
The symmetry of the square implies that there is another invariant subspace that 
intersects $V$ in $\spn(\{\psi_{3,1}\})$, leading to
another BLIS bifurcation.
See \cite{NS} for numerical results of these bifurcating branches, as well as other nearby primary and secondary solution branches.
\end{example}

In the next example, the numerically observed bifurcation point is explained by the BLIS, but is not proven to exist.
\begin{example}
\label{PDEsecondary}
The primary branch predicted by the BLIS in Example~\ref{exa:lambda13} has a secondary bifurcation, 
observed numerically and
shown in \cite[Figure 5]{NS}. 
This bifurcation is not an EBL bifurcation, but it is a BLIS bifurcation as we now explain.
The primary branch is in $W_d$ as implied by the BLIS, but is also observed to be in the invariant fixed point subspace
$$
W_{e,e} := 
\{u \in H \mid u(x, 1-y) =  u(1-x, y) = u(x,y)  \text{ for all } (x,y) \in (0,1)\times(0,1) \}
$$
of functions that are even in both the horizontal and vertical directions.
The secondary bifurcation shown in \cite[Figure 5]{NS} is a BLIS bifurcation with nested subspaces
$W_{e,e} \cap W_d \subsetneq W_{e,e}$.  This is also a BSE bifurcation, since the critical eigenspace is one-dimensional.
Schneider has a description of this type of bifurcation in terms of Fourier coefficients of $u$ in 
\cite{ Schneider2023, Schneider2022}.
\end{example}

The invariant subspace $W_d$ in Example~\ref{exa:lambda13}, based on $\psi_{3,1}$, is just one of an infinite family of invariant subspaces that are not fixed point subspaces based on 
$\psi_{n,m}$ for all cases where $n$ or $m$ is greater than 2.
Similar invariant subspaces are even more prominent for the PDE on the cube.

\begin{example}
\label{cube6D}
In \cite{NSS5} we investigate the PDE $\Delta u + s u + u^3 = 0$ on the cube $(0, \pi)^3$
with $0\,$-Dirichlet boundary conditions.
Here the eigenvalues of the Laplacian are $\lam_{\ell, m, n} = \ell^2 + m^2 + n^2$,
with eigenfunctions $\psi_{\ell, m,n}(x,y,z) = \sin(\ell x) \sin(my)\sin(nz)$.
There are 99 group orbits of fixed point subspaces of the symmetry group of the cube acting on the function space $H$.  These are called symmetry
types, denoted $S_0$ through $S_{98}$.
The paper features the bifurcation of the trivial branch at $s^* = \lam_{1,2,3} = 14$, 
for which $J(s^*)$ has a 6-dimensional critical eigenspace.
\begin{enumerate}
\item[(a)]
This bifurcation has 4 group orbits of EBL branches, indicated in \cite[Figures 23 and 25]{NSS5}.  
These have symmetry types $S_{11}$, $S_{12}$, $S_{22}$, and $S_{23}$.
\item[(b)]
There are also 2 additional BLIS branches that are not EBL branches.  
The invariant subspaces are spaces of functions
with odd reflection symmetry at $z = \pi/3$ and $z = 2 \pi/3$.  These are shown in \cite[Figure 25]{NSS5}
with symmetry type $S_{52}$, with critical eigenfunction $\psi_{1,2,3}$, and the third solution with symmetry type $S_{44}$, with critical eigenfunction $\psi_{1,2,3} + \psi_{2,1,3}$.  
Thus there are 6 group orbits of BLIS branches that bifurcate from $(s, u) = (14, 0)$.
\item[(c)]
In addition, there are 13 group orbits of solution branches that are not BLIS branches.
\end{enumerate}
The branches in Parts~(a) and (b) are simple to follow, since the critical eigenvector is determined by
Theorem~\ref{thm:BLIS}.  
Following BLIS branches is like shooting fish in a barrel.
In contrast, 
the 13 group orbits of non BLIS branches are found by a random search of the 6-dimensional critical eigenspace.
Many hours of computing time were required to find them, and it is possible that some eluded detection.
We are reasonably confident we have found all of the branches, with the help of index theory as
described in \cite{NSS5}.
\end{example}

\section{Numerical Implementation}
\label{sec:Implementation}

We wrote a MATLAB$\textsuperscript{\textregistered}$ program,
available at the GitHub repository~\cite{NSS8_companion}.
It approximates solutions to the system of $n$ equations
\begin{equation}
\label{eqn:original}
F(s, x)_i := f(s, x_i) - (Mx)_i = 0, \quad i \in \{1, 2, \ldots, n \}
\end{equation}
for the components of $x \in \R^n$.  
These are equilibrium solutions of Equation~\eqref{Fofx}, with $k =1$ (so $x_i \in \R$) and the $1 \times 1$ coupling matrix $H = -1$.
The input data for the program includes the weighted adjacency matrix
$M \in \R^{n\times n}$, the nonlinearity $f:\R^2 \rightarrow \R$,
the permutations that commute with $M$,
the $M$-invariant polydiagonal subspaces, and the containment lattice of the
$M$-invariant subspaces.

Figures~\ref{fig:diamondBifDiagram} and \ref{fig:W1BLIS} show bifurcation diagrams
with $M$ chosen to be the Laplacian matrix in Equation~\eqref{diamondL}. We made
text files which input the information shown in Figure~\ref{fig:diamondHasse}.
We chose
\begin{equation}
\label{eqn:standard_f}
f(s, x) = s x + x^3,
\end{equation}
to yield Equation~\eqref{eqn:Fdiamond}.
Other choices of $f$, predefined in our MATLAB$\textsuperscript{\textregistered}$ code,  are
$$
f(s, x) = s x + \alpha x^2 - x^3 \text{ and } f(s,x) = s x + x^3 - \beta x^5,
$$
where $\alpha$ and $\beta$ are parameters which cannot be scaled away.  We do {\em not} require that $f$ is odd in the second variable, but the program treats 
$f$ odd and $f$ not odd differently.  
The $M$-invariant synchrony subspaces are $F$-invariant for all $f$.
The $M$-invariant anti-synchrony subspaces are $F$-invariant if $f$ is odd (see \cite{NiSS}).

We frequently assume $f$ satisfies $f(s,0) = 0$, $D_1f(s, x) = x$, and $D_2f(s, 0) = s$ for all $s$, 
as in the above examples.
This has the advantage that the stability of the trivial branch is easy to compute.
For the trivial branch, with $b_m(s) = 0$ and $W_d = \R^n$,  
the stability matrix
$J(s) = s I_n-M$ is easily analyzed based on the eigenvalues of $M$ (see Equation~\eqref{eqn:Jac}).
Our program works for any smooth $f$, although if $f(s, 0) \neq 0$, then an approximate solution must be supplied to the program as a starting point for the first branch.
For example we have tested
$$
f(s, x) = (s^2 -1)x \pm x^3, \quad f(s, x) = s x + x^3 + 1.
$$

We now describe the general algorithms that, when implemented, 
can generate numerical bifurcation results 
for networks such as those displayed in Figures~\ref{fig:diamondBifDiagram} and~\ref{fig:W1BLIS}.  
There are four essential processes involved: 
\begin{enumerate}
\item Compute the lattice of invariant subspaces.
\item Follow branches, given a starting point, the invariant subspace, and initial direction.
\item Detect bifurcation points, which are the starting points of daughter branches.
\item Use BLIS to find the invariant subspace and initial direction of each daughter branch.
\end{enumerate}

For the first process, 
we start with the matrix $M$ in (\ref{eqn:original}) (typically the graph Laplacian). 
The algorithm described in \cite{NiSS} and implemented on that article's GitHub repository \cite{NiSS_companion}
finds all of the $M$-invariant polydiagonal subspaces 
and the lattice showing which subspaces are contained in which others.
For example, the information in Figure~\ref{fig:diamondHasse} can be computed with this program.

The weighted analytic matrix $M$, invariant subspaces, lattice of inclusion, and $\aut(G)$ information is passed 
to the MATLAB$\textsuperscript{\textregistered}$ program by four text files, as described in the user's 
manual available at \cite{NSS8_companion}.

Given these four text files as inputs, the MATLAB$\textsuperscript{\textregistered}$ code posted on the current paper's GitHub repository \cite{NSS8_companion} computes the bifurcation diagrams showing all of the solutions to Equation~\eqref{eqn:original} that are within the strip $s_\text{min} \leq s \leq s_\text{max}$ and connected
by a series of BLIS bifurcations to the starting solution, usually $x = 0$, $s = s_\text{min}$.
The input parameters are changed by editing the MATLAB$\textsuperscript{\textregistered}$ code.

Branch following is achieved by employing a new version of the so-called 
tGNGA (tangent gradient Newton Galerkin Algorithm) described in \cite{NSS5}.
The added feature is that the computations are done within the invariant subspace. 
That is, we solve Equation~\eqref{FBofw} (with $H = -1$) rather than Equation~\eqref{eqn:original}.  
This speeds up the computations greatly, especially in examples where $n$ is large and $d$ is small.

The input to the tGNGA is the invariant subspace $W_m$ of the branch, a single point $(\tilde s, \tilde x) \in \R \times W_m$, 
and a tangent vector to the branch in  $(\tilde s, \tilde x) \in \R \times W_m$.
Typically, the first branch has $\tilde x = 0$ and $W_m = \{0\}$.  After that,
the initial point is a bifurcation point on a previously computed branch.

The branch following algorithm needs to stop the branch at a bifurcation point when it is the daughter branch.
If the branch continues past the bifurcation point, then it can potentially continue past another bifurcation point
and cause an infinite loop in the MATLAB$\textsuperscript{\textregistered}$ program.

Detecting bifurcation points on a mother branch in $\R \times W_m$ is done using the BLIS.
The Morse Index, which is a non-negative integer, is an important feature in our 
GNGA \cite{NSS3, NSS5, NS}.  It is replaced by a vector of non-negative integers in our new code.
Let $\mathcal D $ be the set of $F$-invariant subspaces of System~\eqref{eqn:original} that contain $W_m$.  As described in \cite{NiSS}, $\mathcal D$ is different depending 
on whether or not $f$ is odd.
These are the possible daughter subspaces.
At each computed point on the mother branch, 
the eigenvalues of the Jacobian, restricted to each $W \in \mathcal D$, are computed,
and a list of the number of eigenvalues with negative real part is recorded.
This is the signature list.
Recall that the matrix $M$ in Equation~\eqref{eqn:original} need not be symmetric, 
so the eigenvalues of the Jacobian are not necessarily real.
If any of the numbers on the signature list changes between two points, the secant method is
used to find the point where an eigenvalue has zero real part.
At a fold point, the signature of the Jacobian, restricted to $W_m$, changes by 1 and this is not
a bifurcation point.
If any other number on the signature list changes by 1, the point where the corresponding
eigenvalue is 0 is a BLIS bifurcation point, and the corresponding critical eigenvector is computed.  

The invariant subspace of the daughter branch is
the smallest $W \in \mathcal D$ that contains the critical eigenvector.
Care must be taken to follow only one daughter branch if several of them
are related by the $\aut(G) \times \Z_2$ symmetry.
Then the bifurcation point, the invariant subspace, and the critical eigenvector are put into a 
queue to start the daughter branch with the tGNGA.
We do not check Condition~(b) in Theorem~\ref{thm:BLIS}.  
This condition is generically true, and 
it is observed to be true in all of the bifurcations we have computed. 
In this way, we follow one branch in each group orbit of branches that are connected to the first branch
(usually the trivial branch) by a sequence of BLIS bifurcations.
%
%

\section{Conclusion}
We have presented the Bifurcation Lemma for Invariant Subspaces (BLIS), which describes most generic
bifurcations of solutions to $F(s, x) = 0$, for $F: \R \times X \to X$.  
The space $X$ can be the Euclidean space $\R^n$, or more generally a Banach space.  Thus, the
theory applies to PDE.
The theory of Bifurcation from Simple Eigenvalues (BSE)
describes the simplest case.
The BLIS is BSE applied to nested invariant subspaces of $F$.  
The BLIS allows the analysis of some bifurcations with multiple 0 eigenvalues.
The Equivariant Branching Lemma (EBL) has been a powerful tool for describing bifurcations
where high multiplicity of a zero eigenvalue is caused by the symmetry of the system.
The EBL is a special case of the BLIS.
Many coupled cell networks have invariant subspaces beyond those caused by symmetry,
and the BLIS can describe bifurcations in such systems.
We have written a freely available MATLAB$\textsuperscript{\textregistered}$ program that computes the solution branches
in a class of coupled cell networks.

While the BLIS does not describe all steady state bifurcations, in many examples most or all of the bifurcating branches are predicted by this simple theorem which is easy to apply.
\bibliographystyle{plain} 
\bibliography{nss8}

\def\cprime{$'$}
\begin{thebibliography}{10}

\bibitem{AD2021}
Manuela Aguiar and Ana Dias.
\newblock Synchrony and antisynchrony in weighted networks.
\newblock {\em SIAM Journal on Applied Dynamical Systems}, 20(3):1382--1420,
  2021.

\bibitem{Aguiar&Dias}
Manuela A.~D. Aguiar and Ana Paula~S. Dias.
\newblock The lattice of synchrony subspaces of a coupled cell network:
  characterization and computation algorithm.
\newblock {\em J. Nonlinear Sci.}, 24(6):949--996, 2014.

\bibitem{cicogna}
G.~Cicogna.
\newblock Symmetry breakdown from bifurcation.
\newblock {\em Lett. Nuovo Cimento (2)}, 31(17):600--602, 1981.

\bibitem{CR}
Michael~G. Crandall and Paul~H. Rabinowitz.
\newblock Bifurcation from simple eigenvalues.
\newblock {\em J. Functional Analysis}, 8:321--340, 1971.

\bibitem{FranciGolubitskyStewart}
Alessio Franci, Martin Golubitsky, Ian Stewart, Anastasia Bizyaeva, and
  Naomi~E. Leonard.
\newblock Breaking indecision in multiagent, multioption dynamics.
\newblock {\em SIAM Journal on Applied Dynamical Systems}, 22(3):1780--1817,
  2023.

\bibitem{bifGraph}
Karin Gatermann.
\newblock Computation of bifurcation graphs.
\newblock In {\em Exploiting symmetry in applied and numerical analysis ({F}ort
  {C}ollins, {CO}, 1992)}, volume~29 of {\em Lectures in Appl. Math.}, pages
  187--201. Amer. Math. Soc., Providence, RI, 1993.

\bibitem{gilbarg2001elliptic}
D.~Gilbarg and N.S. Trudinger.
\newblock {\em Elliptic Partial Differential Equations of Second Order}.
\newblock Classics in Mathematics. Springer Berlin Heidelberg, 2001.

\bibitem{GT}
David Gilbarg and Neil~S. Trudinger.
\newblock {\em Generalized Solutions and Regularity}, pages 177--218.
\newblock Springer Berlin Heidelberg, Berlin, Heidelberg, 2001.

\bibitem{GL_SBL}
Martin Golubitsky and Reiner Lauterbach.
\newblock Bifurcations from synchrony in homogeneous networks: Linear theory.
\newblock {\em SIAM Journal on Applied Dynamical Systems}, 8(1):40--75, 2009.

\bibitem{GSvol1}
Martin Golubitsky and David~G. Schaeffer.
\newblock {\em Singularities and groups in bifurcation theory. {V}ol. {I}},
  volume~51 of {\em Applied Mathematical Sciences}.
\newblock Springer-Verlag, New York, 1985.

\bibitem{GSbook2023}
Martin Golubitsky and Ian Stewart.
\newblock {\em Dynamics and Bifurcation in Networks: Theory and Applications of
  Coupled Differential Equations}.
\newblock Society for Industrial and Applied Mathematics, Philadelphia, PA,
  2023.

\bibitem{GSS}
Martin Golubitsky, Ian Stewart, and David~G. Schaeffer.
\newblock {\em Singularities and groups in bifurcation theory. {V}ol. {II}},
  volume~69 of {\em Applied Mathematical Sciences}.
\newblock Springer-Verlag, New York, 1988.

\bibitem{KameiLattice}
Hiroko Kamei and Peter J.~A. Cock.
\newblock Computation of balanced equivalence relations and their lattice for a
  coupled cell network.
\newblock {\em SIAM J. Appl. Dyn. Syst.}, 12(1):352--382, 2013.

\bibitem{NSS3}
John~M. Neuberger, N{\'a}ndor Sieben, and James~W. Swift.
\newblock Automated bifurcation analysis for nonlinear elliptic partial
  difference equations on graphs.
\newblock {\em Internat. J. Bifur. Chaos Appl. Sci. Engrg.}, 19(8):2531--2556,
  2009.

\bibitem{NSS5}
John~M. Neuberger, N\'andor Sieben, and James~W. Swift.
\newblock Newton's method and symmetry for semilinear elliptic {PDE} on the
  cube.
\newblock {\em SIAM J. Appl. Dyn. Syst.}, 12(3):1237--1279, 2013.

\bibitem{invariantWEB}
John~M. Neuberger, N{\'a}ndor Sieben, and James~W. Swift.
\newblock {I}nvariant {S}ynchrony {S}ubspaces of {S}ets of {M}atrices
  (companion web site), 2019.
\newblock \texttt{https://jan.ucc.nau.edu/ns46/invariant}.

\bibitem{NSS7}
John~M. Neuberger, N{\'a}ndor Sieben, and James~W. Swift.
\newblock Invariant synchrony subspaces of sets of matrices.
\newblock {\em SIAM J. Appl. Dyn. Syst.}, 19(2):964--993, 2020.

\bibitem{NSS8_companion}
John~M. Neuberger, N\'andor Sieben, and James~W. Swift.
\newblock Github repository, \url{https://github.com/jwswift/BLIS/}, 2023.

\bibitem{NS}
John~M. Neuberger and James~W. Swift.
\newblock Newton's method and {M}orse index for semilinear elliptic {PDE}s.
\newblock {\em Internat. J. Bifur. Chaos Appl. Sci. Engrg.}, 11(3):801--820,
  2001.

\bibitem{NiSS_companion}
Eddie Nijholt, N\'andor Sieben, and James~W. Swift.
\newblock Github repository,
  \url{https://github.com/jwswift/Anti-Synchrony-Subspaces/}, 2022.

\bibitem{NiSS}
Eddie Nijholt, N{\'a}ndor Sieben, and James~W. Swift.
\newblock Invariant synchrony and anti-synchrony subspaces of weighted
  networks.
\newblock {\em J. Nonlinear Sci.}, 33(4):38, 2023.
\newblock Id/No 63.

\bibitem{Schneider2023}
I.~Schneider and J.-Y. Dai.
\newblock Symmetry groupoids for pattern-selective feedback stabilization of
  the {C}hafee-{I}nfante equation.
\newblock {\em Chaos}, 33(7):Paper No. 073141, 9, 2023.

\bibitem{Schneider2022}
Isabelle Anne~Nicole Schneider.
\newblock {\em Symmetry Groupoids in Dynamical Systems}.
\newblock Habilitation, 2022.

\bibitem{soares}
Pedro Soares.
\newblock Synchrony branching lemma for regular networks.
\newblock {\em SIAM J. Appl. Dyn. Syst.}, 16(4):1869--1892, 2017.

\bibitem{vander}
A.~Vanderbauwhede.
\newblock {\em Local bifurcation and symmetry}, volume~75 of {\em Research
  Notes in Mathematics}.
\newblock Pitman (Advanced Publishing Program), Boston, MA, 1982.

\end{thebibliography}

\end{document}